\documentclass[journal,twoside]{IEEEtran}

\usepackage{amsmath,amssymb,amsfonts,mathrsfs}
\usepackage{graphicx}
\usepackage{bm}
\usepackage{algorithm,algpseudocode}
\usepackage{hyperref,cite}
\usepackage{textcomp}
\usepackage{xfrac}
\usepackage{multirow}
\usepackage{arydshln}
\usepackage{color}

\usepackage[dvips]{epsfig}
\usepackage{subfig}
\usepackage{epstopdf}

\usepackage{amsthm}

\usepackage{etoolbox}

\newtheorem{theorem}{Theorem}
\newtheorem{lemma}{Lemma}
\newtheorem{corollary}{Corollary}

\newtheorem{definition}{Definition}
\newtheorem{assumption}{Assumption}
\newtheorem{remark}{Remark}

\newcommand{\rank}{{\mathrm {rank}}}
\newcommand{\diag}{{\mathrm {diag}}}
\newcommand{\re}{{\rm Re}}
\newcommand{\im}{{\rm Im}}
\newcommand{\Tr}{\mathrm{Tr}}

\newcommand{\st}{\mathrm{subject~to}}
\newcommand{\diff}{\mathop{}\mathrm{d}}
\newcommand{\mi}{\mathsf{j}}

\newcommand{\nuclearnorm}[1]{\left\lVert#1\right\rVert_*}

\newcommand{\bd}{\mathbf{d}}
\newcommand{\be}{\mathbf{e}}
\newcommand{\bi}{\mathbf{i}}

\newcommand{\bp}{\mathbf{p}}
\newcommand{\bq}{\mathbf{q}}

\newcommand{\bv}{\mathbf{v}}
\newcommand{\bx}{\mathbf{x}}

\newcommand{\bz}{\mathbf{z}}

\newcommand{\bA}{\mathbf{A}}
\newcommand{\bB}{\mathbf{B}}

\newcommand{\bE}{\mathbf{E}}
\newcommand{\bG}{\mathbf{G}}
\newcommand{\bH}{\mathbf{H}}

\newcommand{\bM}{\mathbf{M}}

\newcommand{\bU}{\mathbf{U}}

\newcommand{\bW}{\mathbf{W}}
\newcommand{\bX}{\mathbf{X}}
\newcommand{\bY}{\mathbf{Y}}

\newcommand{\tbU}{\tilde{\mathbf{U}}}

\newcommand{\bSigma}{\bm{\Sigma}}
\newcommand{\bLambda}{\bm{\Lambda}}
\newcommand{\bmu}{\bm{\mu}}
\newcommand{\bnu}{\bm{\nu}}
\newcommand{\bmeta}{\bm{\eta}}
\newcommand{\tbmeta}{\tilde{\bmeta}}

\newcommand{\tbX}{\tilde{\mathbf{X}}}
\newcommand{\tbx}{\tilde{\mathbf{x}}}
\newcommand{\tbv}{\tilde{\mathbf{v}}}
\newcommand{\myexp}{\mathrm{e}}

\newcommand{\cB}{{\mathcal B}}

\newcommand{\cG}{{\mathcal G}}

\newcommand{\cL}{{\mathcal L}}
\newcommand{\cM}{{\mathcal M}}
\newcommand{\cN}{{\mathcal N}}

\newcommand{\cT}{{\mathcal T}}

\DeclareMathOperator*{\argmin}{arg\,min}
\DeclareMathOperator*{\mini}{\mathrm{minimize}}
\DeclareMathOperator*{\maxi}{\mathrm{maximize}}

\begin{document}

\title{Conic Relaxations for Power System State Estimation with Line Measurements}

\author{Yu Zhang, Ramtin Madani, and Javad Lavaei
\thanks{Yu Zhang and Javad Lavaei are  with the Department of Industrial Engineering and Operations Research,
University of California, Berkeley (yuzhang49@berkeley.edu and lavaei@berkeley.edu).
Ramtin Madani is with the Department of Electrical Engineering, University of Texas at Arlington (ramtin.madani@uta.edu).
This work was supported by the DARPA Young Faculty Award, ONR Young Investigator Program Award,
AFOSR Young Investigator Research Program, NSF CAREER Award 1351279, and NSF ECCS Award 1406865.
}}

\markboth{}%
{}
\maketitle


\begin{abstract}
This paper deals with the non-convex power system state estimation (PSSE) problem,
which plays a central role in the monitoring and operation of electric power networks.
Given a set of noisy measurements, PSSE aims at estimating the vector of complex voltages at all buses of  the network. This is a challenging task due to the inherent nonlinearity of power flows, for which existing methods  lack  guaranteed convergence and theoretical analysis.
Motivating by these limitations, we propose a novel convexification framework for the PSSE using semidefinite programming (SDP) and second-order cone programming (SOCP) relaxations.
We first study a related power flow (PF) problem as the noiseless counterpart, which is cast as a constrained minimization program by adding a suitably designed objective function.
We study the performance of the proposed framework in the case where the set of measurements includes:
(i) nodal voltage magnitudes, and (ii) branch active power flows over a spanning tree of the network.
It is shown that  the SDP and SOCP relaxations both recover the true PF solution
as long as the voltage angle difference across each line of the network is not too large
(e.g., less than $90^{\circ}$ for lossless networks).
By capitalizing on this result, penalized SDP and SOCP problems are designed to solve the PSSE,
where a penalty based on the weighted least absolute value is incorporated for fitting noisy measurements with possible bad data.
Strong theoretical results are derived to quantify the optimal solution of the penalized SDP problem, which is shown to
possess a dominant rank-one component formed by lifting the true voltage vector.
An upper bound on the estimation error is also derived as a function of the noise power,
which decreases exponentially fast as the number of measurements increases.
Numerical results on benchmark systems, including a 9241-bus European system,
are reported to corroborate the merits of the proposed convexification framework.
\end{abstract}

\begin{IEEEkeywords}
Power system state estimation, power flow analysis, convex relaxations, semidefinite programming, tree decomposition.
\end{IEEEkeywords}

\section*{Nomenclature}

\addcontentsline{toc}{section}{Nomenclature}

\subsection{Sets and numbers}

\begin{IEEEdescription}[\IEEEusemathlabelsep\IEEEsetlabelwidth{$\cN_p$, $\cN_q$}]

\item[$\cN$, $N$] Set and number of buses.
\item[$\cL$, $L$] Set and number of power lines.
\item[$\cM$, $M$] Set and number of measurements.
\item[$\cN_p$, $\cN_q$]  Sets of active and reactive power injection measurements.
\item[$\mathcal{D}$] Set of all dual SDP certificates.
\end{IEEEdescription}

\subsection{Input signals and constants}

\begin{IEEEdescription}[\IEEEusemathlabelsep\IEEEsetlabelwidth{$\mathbf{Y}$, $\mathbf{Y}_f$, $\mathbf{Y}_t$}]
\item[$\bv$, $\bi$] $N$-dimensional complex vectors of nodal voltages (state of the system) and current injections.
\item[$\bp$, $\bq$] $N$-dimensional real vectors of net injected active and reactive powers.
\item[$\bi_{f}$, $\bi_{t}$] $L$-dimensional complex vectors of current injections at the \emph{from} and \emph{to} ends of all branches.
\item[$\mathbf{Y}$, $\mathbf{Y}_f$, $\mathbf{Y}_t$] Matrices of nodal admittance, \emph{from} branch admittance, and \emph{to} branch admittance.
\item[$\bY_{l,p_{f}}$, $\bY_{l,p_{t}}$] Coefficient matrices corresponding to active power flow measurements at the \emph{from} and \emph{to} ends over the $l$-th branch.
\item[$\bY_{l,q_{f}}$, $\bY_{l,q_{t}}$] Coefficient matrices corresponding to reactive power flow measurements at the \emph{from} and \emph{to} ends over the $l$-th branch.
\item[$\bM_0$] Designed coefficient matrix in the objective.
\item[$\bM_j$] Coefficient matrix corresponding to the $j$-th measurement.
\item[$\bz$]  $M$-dimensional real vector collecting all measurements.
\item[$|v_k|$, $\measuredangle v_{k}$] Voltage magnitude and angle at the $k$-th bus.
\item[$\measuredangle y_{st}$] Angle of the branch $(s,t)$ line admittance.
\item[$\eta_j$, $\sigma_j$] Additive noise and positive weight of the $j$-th measurement.
\item[$\rho$] Positive weight trading off the data fitting cost and the designed linear regularizer.
\item[$\zeta$] Defined root-mean-square estimation error of the obtained optimal SDP solution.

\end{IEEEdescription}

\subsection{Variables and functions}

\begin{IEEEdescription}[\IEEEusemathlabelsep\IEEEsetlabelwidth{$f_{\mathrm{WLAV}}(\cdot)$z}]
\item[$\bX$, $\bH$] $N \times N$  primal and dual matrix variables.
\item[$\boldsymbol{\mu}$] $M$-dimensional real vector of Lagrange multipliers.
\item[$\boldsymbol{\nu}$] $M$-dimensional real vector of slack variables.
\item[$f_{\mathrm{WLAV}}(\cdot)$] Weighted least absolute value cost.
\item[$f_{\mathrm{WLS}}(\cdot)$] Weighted least squares cost. 

\end{IEEEdescription}

\section{Introduction}

An electrical grid infrastructure is operated for delivering electricity from
power generators to consumers via interconnected transmission and distribution networks.
Accurately determining the operating point and estimating the underlying state of the system are of paramount importance for the reliable and economic operation of power networks.
Power flow analysis and power system state estimation play indispensable roles in the planning and monitoring of the power grid. The solutions of these two problems are used for many optimal resource allocation problems such as unit commitment, optimal power flow (OPF), security-constrained OPF, and network reconfiguration \cite{Wollenberg13,GG13}.

\subsection{Power Flow Analysis} \vspace{0mm}

The power flow (PF) problem is a numerical analysis of the steady-state electrical power flows,
which serves as a necessary prerequisite for future system planning.
Specifically, having measured the voltage magnitudes and injected active/reactive powers at certain buses,
 the PF problem aims to find the unknown voltage magnitude and phase angle at each bus of a power network.
Using the obtained voltage phasors and the network admittances, line power flows can then be determined for the entire system.
The calculation of power flows is essentially equivalent to solving a set of quadratic equations obeying the laws of physics.
Solving a system of nonlinear polynomial equations is NP-hard in general.
Be\'{z}out's theorem asserts that a well-behaved system can have exponentially many solutions \cite{Hartshorne77}.
Upper bounds on the number of PF solutions have been analyzed in the recent work \cite{MolzahnACC16} and the references therein.
When it comes to the feasibility of AC power flows, it is known that this problem is NP-hard for both transmission and distribution networks \cite{Bienstock15,Lehmann16}.

For solving the PF problem, many iterative methods such as the Newton-Raphson method and Gauss-Seidel algorithms have been extensively studied over the last few decades \cite{Bergen00}.
The Newton-Raphson method features quadratic convergence whenever the initial point is sufficiently close to the solution \cite{Tinney67,Stott74}.
Nevertheless, a fundamental drawback of various Newton-based algorithms is that there is no convergence guarantee in general. By leveraging advanced techniques in complex analysis and algebraic geometry,
sophisticated tools have been developed for solving PF, including holomorphic embedding load flow
and numerical polynomial homotopy continuation  \cite{Trias12,Li03}.
However, these approaches  involve costly computations, and are generally not suitable for large-scale power systems.
Using the theory of monotone operators and moment relaxations, the papers \cite{DJcdc15} and \cite{DJ15} identify
a ``monotonicity domain'', within which it is possible to efficiently find the PF solutions or certify their non-existence.
A review on recent advances in computational methods for the PF equations can be found in \cite{Mehta15a}.

Facing the inherent challenge of non-convexity, convex relaxation techniques have been recently developed for finding the PF solutions \cite{Madani15CDC}.
More specifically, a class of convex programs is proposed to solve the PF problem in the case where the solution belongs to a recovery region that contains voltage vectors with small angles.
The proposed convex programs are in the form of semidefinite programming (SDP), where a convex objective
is designed as a surrogate of the rank-one constraint to guarantee the exactness of the SDP relaxation.

\subsection{Power System State Estimation}

Closely related to the PF problem, the power system state estimation (PSSE) problem plays a key  role for grid monitoring.
System measurements are acquired through the supervisory control and data acquisition (SCADA) systems,
as well as increasingly pervasive phasor measurement units (PMUs).
Given these noisy measurements, the PSSE task aims at estimating the complex voltage at each bus,
and determining the system's operating conditions.
The PSSE is traditionally formulated as a nonlinear least-squares (LS) problem,
which is commonly solved by the Gauss-Newton algorithm in practice \cite{Monticelli12,Abur04}.
The algorithm is based on a sequence of linear approximations of the nonlinear residuals.
A descent direction is obtained at each iteration by minimizing the sum of squares of the linearized residuals.
However, the Gauss-Newton algorithm has no guaranteed convergence  in general.
Furthermore, a linear search must be carefully carried out for the damped Gauss-Newton method.
The widely-adopted Levenberg-–Marquardt algorithm finds only a local optimum of the nonlinear LS problem,
and may still be slow for large residual or highly nonlinear problems \cite{Bjorck96, Mascarenhas14}.

For a linear regression model, the classic Gauss-Markov theorem states that
if the additive noises are uncorrelated with mean zero and homoscedastic with finite variance,
then the ordinary least squares estimator (LSE) of the unknown parameters is the best linear unbiased estimator (BLUE)
that yields the least variance estimates. The generalized LSE should be applied when the noise covariance matrix is positive definite \cite{Bjorck96}. The work \cite{Zyskind69} shows that even when the noise covariance matrix is singular, the BLUE can be found by utilizing
its pseudo-inverse in the generalized normal equations.
Analytic solutions of the BLUE and the minimum variances of the estimates are available for the linear model.
In addition, minimum variance unbiased estimator (MVUE) and Bayesian-based estimators are studied
in \cite{Amini14} and \cite{Amini15}. It is well known that when the linear measurements are normally distributed, the LSE coincides with the maximum-likelihood estimator (MLE).
However, LSE for the PSSE problem may not possess these attractive properties due to the inherently nonlinear measurements. There are several issues involved from both optimization and statistical perspectives:
\begin{itemize}
\item The problem of nonlinear LS estimation is generally nonconvex, which can have multiple local solutions.
      Hence, finding a globally optimal solution is challenging.
\item Newton-based iterative algorithms are sensitive to the initialization and lack a guaranteed convergence.
      They may converge to a stationary point. It is nevertheless not easy to interpret that point,
      and quantify its distance relative to the true unknown state of the system.
\item Even if a global solution can be obtained, the nonlinear LSE may not correspond to the MVUE.
      When the noises are not from the exponential family of distributions, the LSE is different from the MLE in general.
\item The LSE is vulnerable to the presence of outliers primarily due to its uniformly squaring, which makes data with
      large residuals have a significant influence on the fitted model.
\end{itemize}

To deal with bad data, the weighted least absolute value (WLAV) function is proposed as the data fitting cost in \cite{Irving78, Kotiuga82},
for which efficient algorithms are developed in \cite{Singh94,Abur91}.
The work \cite{Celik92} presents linear transformations to mitigate the deteriorating effect of ``leverage points'' on the WLAV estimator.
Robust or distributed PSSE has also been developed in the papers \cite{Irving08, Kekatos13, Conejo07, Minot16}.
The state estimation problem with line flow measurements using an iterative algorithm is studied in \cite{Dopazo70} and \cite{Dopazo70_2},
where complex power flows over all transmission lines and at least one voltage phasor are assumed to be measured to achieve the necessary redundancy for the solution of the problem. The performance of these selected measurements and the proposed algorithm tested on the Ontario hydro power system are reported in \cite{Porretta73}. Heuristic optimization techniques are also utilized for PSSE in \cite{Naka03,Lee08}.

Intensive studies of the SDP relaxation technique for solving fundamental problems in power networks have been springing up due to the pioneering papers \cite{Bai08}, \cite{lavaei2011_1} and \cite{LL2012_1}.
The work \cite{LL2012_1} develops an SDP relaxation for finding a global minimum of the OPF problem.
A sufficient and necessary condition  is provided to guarantee a zero duality gap, which is satisfied by several benchmark systems.
From the perspective of the physics of power systems, the follow-up papers \cite{SLa12} and \cite{sojoudi2014exactness} develop theoretical results to support the success of the
SDP relaxation in handling the non-convexity of OPF.
The papers \cite{madani2013convex} and \cite{madani2014promises} develop a graph-theoretic SDP framework
for finding a near-global solution whenever the SDP relaxation fails to find a global minimum.
Recent advances in the convex relaxation of the OPF problem are summarized in the tutorial papers \cite{Low2014_1} and \cite{Low2014_2}.

The paper \cite{Zhu11} initializes the idea of solving the PSSE problem via the SDP relaxation.
When the SDP  solution is not rank one, its principal eigenvector is used to recover approximate voltage phasors.
 The work \cite{Weng12} suggests generating a ``good'' initial point from the SDP optimal solution
to improve the performance of Newton's method, while a nuclear norm regularizer is used to promote a low-rank solution in \cite{Weng15}.
Distributed or online PSSE using the SDP relaxation can be found in \cite{Weng13,Zhu14,KimGG14}.
However, in the literature there is a lack of theoretical analysis on 
the quality of the SDP optimal solution for estimating the complex voltages. 
Hence, to the best of our knowledge, this is still an intriguing open problem.

The aforementioned grand challenges of the PSSE problem motivate us to revisit the design of a high-performance estimator with finite measurements.
The novelty and main contributions of the present work are outlined in the ensuing subsection.

\subsection{Contributions}
In this paper, we start with a PF problem that can be regarded as the noiseless counterpart of PSSE.
In contrast to the standard setup with only nodal measurements at the PV, PQ and slack buses,
one objective of this work is to investigate the effect of branch flow measurements on reducing the computational complexity of  the PF problem.
Motivated by the work \cite{Madani15CDC},
we contrive a convex optimization framework for the PF problem using  SDP and second-order cone programming (SOCP) relaxations.
It is shown that  the proposed conic relaxations are both always exact if: (i) the set of measurements includes the nodal voltage magnitude at each bus and line active power flows over a spanning tree of the power network, and (ii) the line phase voltage differences are not too large (e.g., less than $90^{\circ}$ for lossless networks).

By building upon the proposed convexification framework for the PF problem, we develop a penalized convex program for solving the PSSE problem.
In addition to an $\ell_1$ norm penalty that is robust to outliers in the measurements,
the objective function of the penalized convex problem features a linear regularization term
whose coefficient matrix can be systematically designed according to the meter placements.
We present a theoretical result regarding the quality of the optimal solution of the convex program.
It is shown that the obtained optimal solution has a dominant rank-one matrix component,
which is formed by lifting the vector of true system state. The distance between the
solution of the penalized convex problem and the correct rank-one component is quantified as a function of the noise level.
An upper bound of the tail probability of this distance is further derived,
which also implies the correlation between the quality of the estimation and the number of measurements.

The effort of this paper is mainly on the scenario where the measurements include nodal voltage magnitudes and branch active power flows.
However, the developed mathematical framework is rather general and could be adopted to study the PSSE problem with other types of measurements.

\subsection{Notations}

Boldface lower (upper) case letters represent column vectors (matrices), and calligraphic letters stand for sets.
The symbols $\mathbb{R}$ and $\mathbb{C}$ denote the sets of real and complex numbers, respectively.
$\mathbb{R}^{N}$ and $\mathbb{C}^{N}$ denote the spaces of $N$-dimensional real and complex vectors, respectively.
$\mathbb{S}^N$ and $\mathbb{H}^N$ stand for the spaces of $N\times N$ complex symmetric and Hermitian matrices, respectively.
The symbols $(\cdot)^{\top}$ and $(\cdot)^{*}$ denote the transpose and conjugate transpose of a vector/matrix.
$\re(\cdot)$, $\im(\cdot)$, $\rank(\cdot)$,  $\Tr(\cdot)$,  and $\mathrm{null}(\cdot)$  denote the real part,
imaginary part, rank, trace, and null space of a given scalar or matrix.
$\|\mathbf{a}\|_2$, $\|\bA\|_F$, and $\nuclearnorm{\bA}$ denote the Euclidean norm of the vector $\mathbf{a}$,
the Frobenius norm and the nuclear norm of the matrix $\bA$, respectively.
The relation $\bX \succeq \mathbf{0}$ means that the matrix $\bX$ is Hermitian positive semidefinite.
The $(i,j)$ entry of $\bX$ is given by $X_{i,j}$. $\mathbf{I}_{N}$ denotes the $N\times N$ identity matrix.
The symbol $\diag(\bx)$ denotes a diagonal matrix whose diagonal entries are given by the vector $\bx$, while
$\diag(\bX)$ forms a column vector by extracting the diagonal entries of the matrix $\bX$.
The imaginary unit is denoted by $\mi$.
The expectation operator and the probability measure are denoted by $\mathbb{E(\cdot)}$ and $\mathbb{P}(\cdot)$, respectively.
The notations $\measuredangle x$ and $\lvert x\rvert$ denote the angle and  magnitude of a complex number $x$.
The notation $\bX[\mathcal{S}_1,\mathcal{S}_2]$ denotes the submatrix of $\bX$ whose rows and columns are
chosen from the  given index sets $\mathcal{S}_1$ and $\mathcal{S}_2$, respectively.

\section{Preliminaries}

\subsection{System Modeling}\label{sec:systmodel}
Consider an electric power network represented by a graph $\cG = (\cN,\cL)$,
where $\cN := \{1,\ldots,N\}$ and $\cL:= \{1,\ldots,L\}$ denote the sets of buses and branches, respectively.
Let $v_k \in \mathbb{C}$ denote the nodal complex voltage at bus $k\in\mathcal N$,
whose magnitude and phase angle are given as $|v_k|$ and $\measuredangle v_k$.
The net injected complex power at bus $k$ is denoted as $s_k=p_k+q_k\mi$.
Define $s_{lf}=p_{lf}+q_{lf}\mi$ and $s_{lt}=p_{lt}+q_{lt}\mi$
as the complex power injections entering the line $l\in \cL$ through the \emph{from} and \emph{to} ends of the branch. 
Note that the current $i_{l,f}$ and $i_{l,t}$  may not add up to zero due to the existence of transformers and shunt capacitors.
 Denote the admittance of each branch $(s,t)$ of the network as $y_{st}$.
The Ohm's law dictates that
\begin{align}
	\bi = \bY\bv,\quad \bi_{f} = \bY_{f}\bv,\quad \mathrm{and} \quad \bi_{t} = \bY_{t}\bv,
\end{align}
where $\bY = \bG + \mi\bB \in \mathbb{S}^{N}$ is the nodal admittance matrix of the power network, whose real and imaginary parts
are the conductance matrix $\bG$ and susceptance matrix $\bB$, respectively.
Furthermore, $\bY_{f}\in \mathbb{C}^{L\times N}$ and $ \bY_{t} \in \mathbb{C}^{L\times N}$
represent the \emph{from} and \emph{to} branch admittance matrices.
The injected complex power can thus be expressed as $\bp + \bq\mi = \diag(\bv\bv^{*}\bY^{*})$.
Let $\{\be_1,\ldots,\be_N\}$ denote the canonical vectors in $\mathbb{R}^N$. Define
\begin{equation}\label{nodalM}
	\begin{aligned}
		 \bE_{k}  &:= \be_k \be_k^{\top},\quad \bY_{k,p} := \frac{1}{2}(\bY^{*}\bE_{k}+\bE_{k}\bY),\\
		\bY_{k,q} &:= \frac{\mi}{2}(\bE_{k}\bY-\bY^{*}\bE_{k}).
	\end{aligned}
\end{equation}
For each $k\in \cN$, the quantities $|v_k|^2$, $p_k$ and $q_k$ can be written as
\begin{equation}\label{nodalQTY}
\hspace{-1mm}
		|v_k|^2  = \Tr(\bE_{k}\bv\bv^{*}),\
		p_k      = \Tr(\bY_{k,p}\bv\bv^{*}),\
		q_k      = \Tr(\bY_{k,q}\bv\bv^{*}).
\end{equation}
Similarly, the branch active and reactive powers for each line $l\in \cL$ can be expressed as
\begin{equation}\label{branchQTY}
	\begin{aligned}
		p_{l,f}  &= \Tr(\bY_{l,p_{f}}\bv\bv^{*}),\quad
		p_{l,t}  = \Tr(\bY_{l,p_{t}}\bv\bv^{*}) \\
		q_{l,f}  &= \Tr(\bY_{l,q_{f}}\bv\bv^{*}),\quad
		q_{l,t}  = \Tr(\bY_{l,q_{t}}\bv\bv^{*}),
	\end{aligned}
\end{equation}
where the coefficient matrices $\bY_{l,p_{f}},\bY_{l,p_{t}},\bY_{l,q_{f}},\bY_{l,q_{t}} \in \mathbb{H}^N$ are defined
over the $l$-th branch from node $i$ to node $j$ as
\begin{subequations}\label{branchM}
	\begin{align}
		\bY_{l,p_{f}}  &:= \frac{1}{2}(\bY^{*}_f\bd_l\be_{i}^{\top}+\be_{i}\bd_l^{\top}\bY_f) \label{Ylpf} \\
		\bY_{l,p_{t}}  &:= \frac{1}{2}(\bY^{*}_t\bd_l\be_{j}^{\top}+\be_{j}\bd_l^{\top}\bY_t) \\
		\bY_{l,q_{f}}  &:= \frac{\mi}{2}(\be_{i}\bd_l^{\top}\bY_f- \bY^{*}_f\bd_l\be_{i}^{\top}) \\
		\bY_{l,q_{t}}  &:= \frac{\mi}{2}(\be_{j}\bd_l^{\top}\bY_t-\bY^{*}_t\bd_l\be_{j}^{\top}),
	\end{align}
\end{subequations}
where $\{\bd_1,\ldots,\bd_L\}$ is the set of  canonical vectors in $\mathbb{R}^{L}$.

So far, nodal and line measurements of interest have been  expressed as quadratic functions of the complex voltage $\bv$.
The  PF and PSSE problems will be formulated next.

\vspace{-2mm}
\subsection{Convex Relaxation of Power Flow Equations}\label{sec:probform}

The task of the PSSE problem is to estimate the complex voltage vector $\bv$ based on $M$ real measurements:
\begin{align}
z_j = \bv^{*}\bM_j\bv + \eta_j, \quad \forall j \in \cM:=\{1,2,\ldots,M\},
\end{align}
where $\{z_j\}_{j \in \cM}$ are the known measurements, $\{\eta_j\}_{j \in \cM}$ are the possible measurement noises with  known statistical information, and $\{\bM_j\}_{j \in \cM}$ are
 arbitrary measurement matrices that could be  any subset of the Hermitian matrices defined in \eqref{nodalM} and \eqref{branchM}.
The  PF problem is a noiseless version of the PSSE problem.
More specifically, given a total of $M$ noiseless specifications $z_j$ for $j=1,2,\ldots,M$, the goal of PF
is to find the nodal complex voltage vector $\bv$ satisfying all quadratic measurement equations, i.e., 
\begin{subequations}\label{pfp}
\begin{align}
\mathrm{find}\quad &\bv \in \mathbb{C}^N \\
\st\quad  &\bv^{*}\bM_j\bv = z_j,\quad\forall j\in \cM.
\end{align}
\end{subequations}
After setting the phase of the voltage at the slack bus to zero, the problem reduces to $M$ power flow  equations with $2N-1$ unknown real parameters. The classical PF problem corresponds to the case $M=2N-1$, where the measurements are specified at the PV, PQ, and slack buses such that: 
\begin{itemize}
\item 
For each PV  (generator) bus $k$, the active power $p_k$ and the voltage magnitude $|v_k|$ are given.
\item 
For each PQ  (load) bus $k$, the active power $p_k$ and the reactive power $q_k$ are given.
\item For the slack  (reference) bus, the voltage magnitude $|v_{\mathrm{ref}}|$ and the phase angle $\measuredangle v_{\mathrm{ref}}$ are given.
\end{itemize}

Instead of solving the feasibility problem \eqref{pfp} to obtain the voltage vector $\bv$,
consider the  optimization problem
\begin{subequations}\label{PFP2}
\begin{align}
\quad \mini_{\bX \in \mathbb{H}^N, \bv \in \mathbb{C}^N} \quad &\Tr(\bM_0\bX) \\
\st\quad  &\Tr(\bM_j\bX) = z_j,\quad\forall  j\in \cM \\
\quad & \bX = \bv\bv^{*},
\end{align}
\end{subequations}
where its objective function is to be designed later.
Note that the constraint $\bX = \bv\bv^{*}$ can be equivalently replaced by the two conditions $\bX \succeq \mathbf{0}$ and $\rank(\bX) = 1$.
The SDP relaxation of~\eqref{PFP2} is obtained by dropping the rank-one constraint as
\begin{subequations}\label{PF-SDPP}
\begin{align}
\mini_{\bX \in \mathbb{H}^N}\quad &\Tr(\bM_0\bX) \label{PF-SDPP:obj} \\
\st\quad  &\Tr(\bM_j\bX) = z_j,\quad\forall  j\in \cM \label{PF-SDPP:meq} \\
\quad & \bX \succeq \mathbf{0}.\label{PF-SDPP:cone}
\end{align}
\end{subequations}
This relaxation correctly solves \eqref{PFP2} if and only if  it has a unique rank-1 solution $\bX^{\text{opt}}$, in which case $\bv$ can be recovered via the decomposition $\bX^{\text{opt}}=\bv \bv^{*}$.
The dual of \eqref{PF-SDPP} can be obtained as
\begin{subequations}\label{PF-SDPD}
\begin{align}
\maxi_{\bmu \in \mathbb{R}^M}\quad &-\bz^{\top}\bmu \\
\st\quad  &
\mathbf{H}(\boldsymbol{\mu}) \succeq \mathbf{0}, \label{PF-SDPD:PSD}
\end{align}
\end{subequations}
where the vector $\bz:=[z_1,\ldots,z_M]^{\top}$ collects all the available measurements,
$\bmu = [\mu_1,\ldots,\mu_M]^{\top} $ is the Lagrangian multiplier vector associated with
the linear equality constraints \eqref{PF-SDPP:meq}, and  the dual matrix function $\mathbf{H}:\mathbb{R}^M\to\mathbb{H}^N$ is defined as
\begin{align}
\mathbf{H}(\boldsymbol{\mu}):=\bM_0 + \sum_{j = 1}^M \mu_j\bM_j.\label{Hdef}
\end{align}
If strong duality holds while the primal and dual problems both attain their solutions, then every pair of optimal primal-dual solutions $(\mathbf{X}^{\mathrm{opt}},\boldsymbol{\mu}^{\mathrm{opt}})$ satisfies the relation $\bH(\boldsymbol{\mu}^{\mathrm{opt}}) \bX^{\mathrm{opt}}= \mathbf{0}$, due to the complementary slackness.
Hence,  if $\rank(\bH(\boldsymbol{\mu}^{\mathrm{opt}})) = N-1$ holds,  then we have the inequality $\rank(\bX^{\mathrm{opt}}) \leq 1$ such that the SDP relaxation can recover a solution of the PF problem.

\begin{definition}[SDP recovery]
It is said that the SDP relaxation problem \eqref{PF-SDPP} recovers the voltage vector $\bv \in \mathbb{C}^N$ if $\bX=\bv\bv^{*}$ is the unique solution of \eqref{PF-SDPP} for some  input $\bz \in \mathbb{R}^M$.
\end{definition}

\begin{definition}[Dual certificate]
\label{dual_cer_def}
A vector $\boldsymbol{\mu}\in\mathbb{R}^M$ is regarded as a dual SDP certificate for the voltage vector $\mathbf{v}\in\mathbb{C}^N$ if it satisfies the following three properties:
\begin{align}\label{dual_cer}
\!\!\!\!
\mathbf{H}(\boldsymbol{\mu})\succeq \mathbf{0},\quad
\mathbf{H}(\boldsymbol{\mu})\mathbf{v}= \mathbf{0},\quad
\mathrm{rank}( \mathbf{H}(\boldsymbol{\mu}) )= N-1.\!\!
\end{align}
Denote the set of all dual SDP certificates for the voltage vector $\mathbf{v}$ as $\mathcal{D}(\mathbf{v})$.\vspace{-2mm}
\end{definition}

The SDP problem \eqref{PF-SDPP} can be further relaxed by replacing the high-order positive semidefinite constraint \eqref{PF-SDPP:cone} with second-order conic constraints on $2\times 2$ principal sub-matrices of $\mathbf{X}$ corresponding to certain lines of the network. This yields the SOCP relaxation:
\begin{subequations}\label{PF-SOCP}
	\begin{align}
	\mini_{\bX \in \mathbb{H}^N}\quad &\Tr(\bM_0\bX) \label{PF-SOCP:obj} \\
	\st\quad  &\Tr(\bM_j\bX) = z_j,\, &&\forall j\in \cM \label{PF-SOCP:meq} \\
	\quad &
	\begin{bmatrix}
	X_{s,s} & X_{s,t}\\
	X_{t,s} &X_{t,t}
	\end{bmatrix}
	\succeq  \mathbf{0}, &&\forall(s,t) \in \overline\cL,\label{PF-SOCP:cone}
	\end{align}
\end{subequations}
where $\overline \cL$ denotes the set of those edges of the network graph for which
the corresponding entry of $\mathbf{M}_{j}$ is nonzero for at least one index $j \in \{0,1,\ldots,M\}$.


\begin{definition}[SOCP recovery]
	It is said that the SOCP relaxation problem \eqref{PF-SOCP} recovers the voltage vector $\bv \in \mathbb{C}^N$ if there is some  input $\bz \in \mathbb{R}^M$ such that, for every solution $\bX^{\mathrm{opt}}$ of  \eqref{PF-SOCP}, those entries of the matrix $\bX^{\mathrm{opt}}-\bv\bv^{*}$ on the diagonal or corresponding   to the members of $\overline\cL$ are all equal to zero.
\end{definition}

\begin{figure}[t]
	\centering
	{\includegraphics[width=0.24\textwidth]{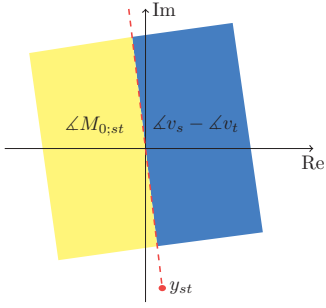}}
	\caption{The demonstration of the angle conditions \eqref{angleMy} and \eqref{angleV}. The acceptable
regions for the voltage phase difference $\measuredangle v_{s}-\measuredangle v_{t}$ (blue open half-space)
and the entry $M_{0;st}$ (yellow open half-space) are shown relative to the  branch admittance $y_{st}$ (red dot). }
	\label{fig:angleCond}
\end{figure}

\vspace{-2mm}

\section{Exact Recovery of Power Flow Solution}

The objective of this section is to show that the SDP problem~\eqref{PF-SDPP} is exact and the correct complex voltage  vector $\bv$ can be recovered for a class of nodal and branch noiseless measurements.
Let $\cG^{\prime} = (\cN,\cL^{\prime})$ denote an arbitrary  subgraph of $\mathcal G$ that contains a spanning tree of $\mathcal G$.
Throughout the rest of this section, we assume that the available measurements consist of:
(i) voltage magnitudes at all buses, and
(ii) active power flow at the ``from'' end of each line of $\cG^{\prime}$. Note that whenever the SDP relaxation is exact for this set of measurements, it remains exact if more measurements are available. Please refer to Corollary~\ref{coro:nodalM} and Remark~\ref{rem:r2} for more details.

The SDP relaxation of \eqref{PFP2} can be expressed as
\begin{subequations}\label{PF-SDPP-specM}
	\begin{align}
	\mini_{\bX \in\mathbb{H}^{N}}\quad &\Tr(\bM_0\bX) \label{PF-SDPP-specM:obj}\\
	\st\quad  &X_{k,k} = |v_k|^2, &&\forall k\in \cN \label{PF-SDPP-specM:meq_Node} \\
			  &\mathrm{Tr}(\mathbf{Y}_{l,p_f}\mathbf{X})=p_{l,f}, &&\forall l\in  \cL^{\prime}
			  \label{PF-SDPP-specM:meq_Branch}\\
			  &\mathbf{X}\succeq \mathbf{0}.
	\end{align}
\end{subequations}
Moreover, the SOCP relaxation of \eqref{PFP2} can be written as:
\begin{subequations}\label{PF-SOCP-specM}
	\begin{align}
	\mini_{\bX \in\mathbb{H}^{N}}\quad &\Tr(\bM_0\bX)\label{PF-SOCP-specM:obj}\\
	\st\quad  &X_{k,k} = |v_k|^2, &&\forall k\in \cN  \label{PF-SOCP-specM:meq_Node}\\
	&\mathrm{Tr}(\mathbf{Y}_{l,p_f}\mathbf{X})=p_{l,f}, &&\forall l\in \cL^{\prime}\label{PF-SOCP-specM:meq_Branch}\\
	\quad &
	\begin{bmatrix}
	X_{s,s} & X_{s,t}\\
	X_{t,s} &X_{t,t}
	\end{bmatrix}
	\succeq  \mathbf{0}, &&\forall(s,t) \in \cL^{\prime}.\label{PF-SOCP-specM:cone}
	\end{align}
\end{subequations}

\begin{definition}[Sparsity graph]
Given a Hermitian matrix $\mathbf{W}\in\mathbb H^N$,
the sparsity graph of $\mathbf{W}$, denoted by $\mathscr{G}(\bW)$, is a simple undirected graph with the vertex set  $\{1,2,\ldots,N\}$ such that every two distinct vertices $i$ and $j$ are connected to each other  if  and only if the $(i,j)$ entry of $\mathbf{W}$ is nonzero.
\end{definition}

\begin{assumption}\label{asmp1}
The edge set of $\mathscr{G}(\mathbf{M}_0)$ coincides with $\cL^{\prime}$ and in addition,
\begin{align} \label{angleMy}
-180^{\circ}<\measuredangle M_{0;st}-\measuredangle y_{st}<0,\quad \forall (s,t)\in\cL^{\prime},
\end{align}
where $M_{0;st}$ denotes the $(s,t)$ entry of $\bM_0$.
Moreover, the solution $\mathbf{v}$ being sought satisfies the relations
\begin{subequations}
\label{angleV}
\begin{align}
\hspace{-0.2cm}0<(\measuredangle v_{s}-\measuredangle v_{t})-\measuredangle y_{st}<180^{\circ},\quad \forall (s,t)\in\cL^{\prime} \label{angleV1}\\
\hspace{-0.2cm}(\measuredangle v_{s}-\measuredangle v_{t})-\measuredangle M_{0;st}\neq 0 \, \ \mathrm{or}\, \ 180^{\circ},\quad \forall (s,t)\in\cL^{\prime}. \label{angleV12}
\end{align}
\end{subequations}
\end{assumption}
To reduce power losses, real-world transmission systems feature low R/X ratios (the ratio of line resistance to reactance).
The angle of the line admittance $\measuredangle y_{st}$ is therefore close to $-90^{\circ}$ \cite[Sec. 3.7]{Weedy12}.
Meanwhile, since the transferred real power is proportional to its corresponding voltage angle difference, 
the number $|\measuredangle v_{s}-\measuredangle v_{t}|$ is typically small due to thermal and stability limits \cite{GG13,Andersson08}.
Hence, the angle condition \eqref{angleV1} is expected to hold. For lossless networks,  \eqref{angleV1} requires each line voltage angle difference to be between $-90^{\circ}$ and $90^{\circ}$, which is a very practical assumption.  The acceptable regions for $\measuredangle v_{s}-\measuredangle v_{t}$
and $M_{0;st}$ are shown in Figure \ref{fig:angleCond}. It can be observed that one convenient choice for the matrix $\bM_{0}$ is to select its entries $M_{0;st}$  as  complex numbers with negative real and imaginary parts.

\begin{lemma}
\label{lem:dualcertf}
Under Assumption~\ref{asmp1}, there exists a dual SDP certificate for the voltage vector $\bv \in \mathbb{C}^N$.
\end{lemma}
\begin{proof}
The proof is provided in  Appendix~\ref{appendix:dualcertf}.
\end{proof}

\begin{theorem}
\label{thm:tightrelax}
Under Assumption~\ref{asmp1}, the SDP relaxation \eqref{PF-SDPP-specM} and the  SOCP relaxation~\eqref{PF-SOCP-specM} both recover the voltage vector $\bv \in \mathbb{C}^N$.
\end{theorem}
\begin{proof}
The proof  is provided in  Appendix~\ref{appendix:tightrelax}.~\end{proof}
To be able to recover a large set of voltage vectors, Theorem~\ref{thm:tightrelax} implies that there are infinitely many choices for the objective function of the SDP relaxation, namely all matrices $\bM_0$ satisfying Assumption~\ref{asmp1}. Now, consider the case with  extra nodal measurements
\begin{subequations}\label{nodalPQ}
\begin{align}
\Tr(\bY_{k,p}\bX) &= p_k, \quad\ \forall k\in \cN_p \\
\Tr(\bY_{k^{\prime},q}\bX) &= q_{k^\prime}, \quad \forall k^{\prime}\in \cN_q.
\end{align}
\end{subequations}
The next corollary shows that the property of the exact relaxation is preserved in presence of
these arbitrary extra power injection measurements. As will be studied later in the paper, the availability of extra measurements seems unnecessary for the PF problem, but  is instrumental in recovering the state of the system in the noisy setup.

\begin{corollary}
\label{coro:nodalM}
Under Assumption \ref{asmp1}, the SDP relaxation \eqref{PF-SDPP-specM} and the SOCP relaxation~\eqref{PF-SOCP-specM} with the
additional constraints of power injection measurements \eqref{nodalPQ}
both recover the voltage vector $\bv \in \mathbb{C}^N$.
\begin{proof}
With extra nodal power measurements, $\mathbf{X}=\mathbf{v}\mathbf{v}^{\ast}$ still remains feasible for both problems. Therefore the corollary comes as a direct result of Theorem~\ref{thm:tightrelax}.
\end{proof}
\end{corollary}

\begin{figure}[t]
\centering
\includegraphics[width=0.27\textwidth]{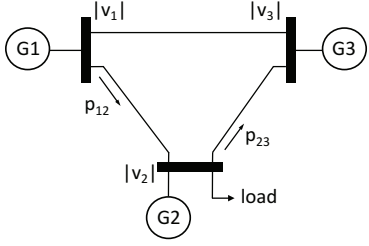}
\caption{A 3-bus power network with the voltage magnitude measurements $|v_1|$, $|v_2|$ and $|v_3|$,
as well as the branch active power measurements $p_{12}$ and $p_{23}$.}
\label{fig:3bus}
\vspace{-0.4cm}
\end{figure}

\vspace{-2mm}

\subsection{Effect of Reactive Power Branch Measurements}

In the preceding section, the exactness of the SDP and SOCP relaxations were studied in the case with the measurement of branch active power flows. In what follows, it will be shown that reactive power line flows do not offer the same benefits as active power measurements.  Assume that, as opposed to the active power flow, the reactive  power flow at the ``from'' end of each branch of $\cG^{\prime}$ is measured .
In this case, Theorem~\ref{thm:tightrelax} still holds if the conditions provided in Assumption~\ref{asmp1} are replaced by:
\begin{subequations}
\begin{align}
	\mathrm{Re}(M_{0;st}y^{\ast}_{st})\neq 0 \quad \mathrm{and} \quad
	\mathrm{Im}(v_sv^{\ast}_t M^{\ast}_{0;st}) &\neq 0 \\
	\mathrm{Re}(v_sv^{\ast}_t y^{\ast}_{st})\mathrm{Re}(M_{0;st}y^{\ast}_{st}) & \leq 0. \label{recond2Q}
\end{align}
\end{subequations}
In contrast to the case with the measurements of $p_{l,f}$, the following two different scenarios must be considered for \eqref{recond2Q}
\begin{itemize}
\item[(i):] if $90^{\circ}<(\measuredangle v_{s}-\measuredangle v_{t})-\measuredangle y_{st}\leq  180^{\circ}$, then $\mathrm{Re}(v_sv^{\ast}_t y^{\ast}_{st})<0$
and  $\mathrm{Re}(M_{0;st}y^{\ast}_{st})>0$, which imply that
\begin{align}
-90^{\circ} \leq \measuredangle M_{0;st}-\measuredangle y_{st}\leq  90^{\circ},
\end{align}
\item[(ii):] if $0 \leq (\measuredangle v_{s}-\measuredangle v_{t})-\measuredangle y_{st} < 90^{\circ} $, then $\mathrm{Re}(v_sv^{\ast}_t y^{\ast}_{st})>0$
and  $\mathrm{Re}(M_{0;st}y^{\ast}_{st})<0$, which imply that
\begin{align}
90^{\circ} \leq \measuredangle M_{0;st}-\measuredangle y_{st} \leq 270^{\circ}.
\end{align}
\end{itemize}
As a result,  $\measuredangle M_{0;st}$ must belong to one of the two complementary intervals
$[\measuredangle y_{st}+90^{\circ}, \measuredangle y_{st}+270^{\circ}]$ and $[\measuredangle y_{st}-90^{\circ}, \measuredangle y_{st}+90^{\circ}]$,
depending on the value of $\measuredangle v_{s}-\measuredangle v_{t}$. Therefore, it is impossible to design the matrix $\bM_{0}$ in advance
without knowing the phase angle difference $\measuredangle v_{s}-\measuredangle v_{t}$.

\begin{remark} \label{rem:r2}
We assume that available measurements include the voltage magnitude at each bus and active line flows over at least a spanning tree of the power network.
Such an assumption is realistic in practical power systems since these two types of measurements are typically provided by the SCADA system with little incremental cost \cite{Korres2011},
while also used for conventional static state estimation algorithms \cite{Phadke08}.
Another source of voltage magnitude measurements comes from the increasing usage of PMUs.
Moreover, the selection of line power flow measurements features several advantages \cite{Dopazo70,Porretta73}:
\begin{itemize}
 \item The spanning tree line flow measurements ensure the network observability \cite{Abur99,WuKK06}.
  \item The line flow measurements can be directly used  for monitoring, which is of practical importance.
  \item Measurements at both ends of lines are very effective in detecting and identifying incorrect data.
  \item The numerical computation is fast and stable, while the results are less sensitive to measurement errors.
\end{itemize}
Nevertheless, the above assumption on the types of measurements  is not essential for the validity of the proposed convexification framework. In other words, this framework can be deployed for arbitrary measurements, but we study its performance under the above assumption. It is worth stressing that, similar to the aforementioned PF problem, additional measurements such as nodal power injections 
can be readily incorporated in our framework for PSSE. 
\end{remark}

\subsection{Three-Bus Example}

Consider the 3-bus power system shown in Figure~\ref{fig:3bus}.
Suppose that the measured signals consist of  the two active power line flows $p_{12}$ and $p_{23}$, as well as the nodal voltage squared magnitudes
$|v_1|^2$, $|v_2|^2$ and $|v_3|^2$. Theorem~\ref{thm:tightrelax} states that the SDP and SOCP relaxation problems~\eqref{PF-SDPP-specM} and \eqref{PF-SOCP-specM} are both able to find the unknown voltage vector $\bv$, using an appropriately designed coefficient matrix $\bM_0$. It turns out that $\bv$ can also be found through a direct calculation. More precisely, one can write
\begin{subequations}\label{eq:ptheta}
\begin{align}
  p_{12} &= \re(v_1(v_1-v_2)^{*}y_{12}^{*})= |v_1|^2\re(y_{12})  \notag \\
         &- |v_1||v_2||y_{12}|\cos(\measuredangle v_{1}-\measuredangle v_{2}-\measuredangle y_{12}) \\
  p_{23} &= \re(v_2(v_2-v_3)^{*}y_{23}^{*}) = |v_2|^2\re(y_{23})\notag \\
      &    - |v_2||v_3||y_{23}|\cos(\measuredangle v_{2}-\measuredangle v_{3}-\measuredangle y_{23}),
\end{align}
\end{subequations}
which yields that
\begin{subequations}\label{eq:vtheta}
\begin{align}
\measuredangle v_{1}-\measuredangle v_{2}&= \arccos\left(\frac{p_{12}-|v_1|^2\re(y_{12})}{|v_1||v_2||y_{12}|}\right) +\measuredangle y_{12} \\
\measuredangle v_{2}-\measuredangle v_{3} &= \arccos\left(\frac{p_{23}-|v_2|^2\re(y_{23})}{|v_2||v_3||y_{23}|}\right)+ \measuredangle y_{23}.
\end{align}
\end{subequations}
Each phase difference $\measuredangle v_{1}-\measuredangle v_{2}$ or $\measuredangle v_{2}-\measuredangle v_{3}$ can have two possible solutions,
but only one of them satisfies the angle condition \eqref{angleV1}. Hence, all complex voltages can be readily recovered.
This argument applies to general power networks. In other words, without resorting to the relaxed problems~\eqref{PF-SDPP-specM} and \eqref{PF-SOCP-specM},
the PF problem considered in this paper can be directly solved by the calculation of phase angles.
However, once the measurements are noisy, the equations \eqref{eq:vtheta} cannot be used because the exact values of the quantities $p_{12}$, $p_{23}$, $|v_1|^2$, $|v_2|^2$ and $|v_3|^2$ are no longer available since they are corrupted by noise.
In contrast, the proposed SDP and SOCP relaxations work in both noiseless and noisy cases. This will be elaborated in the next section.

As a byproduct of the discussion made above, one can obtain the following result.

\begin{corollary}
\label{cor:c2}
The PF problem has a unique solution satisfying  Assumption~\ref{asmp1}. Moreover, this solution can be recovered using the SDP relaxation \eqref{PF-SDPP-specM} and the  SOCP relaxation~\eqref{PF-SOCP-specM}.
\end{corollary}

\vspace{-2mm}
\section{Convexification of State Estimation Problem}

Consider the PSSE as a generalization of the PF problem, where  the measurements are subject to noise. As explained in Corollary~\ref{cor:c2}, the unknown solution $\bold v$ is unique under Assumption~\ref{asmp1}.
To find this solution,  consider the optimization problem:
\begin{subequations}\label{prob:PSSE2}
\begin{align}
\mini_{\bv \in \mathbb{C}^N,\, \bnu \in \mathbb{R}^M}\quad &f(\bnu) \\
\st\quad  &z_j - \bv^{*}\bM_j\bv = \nu_j,\quad \forall j\in \cM, \label{PSSE2:constraint}
\end{align}
\end{subequations}
where $\bnu:=[\nu_1,\ldots,\nu_M]^{\top}$ and the function $f(\cdot)$ quantifies the estimation criterion.
Common choices of $f(\cdot)$ are the weighted $\ell_1$ and $\ell_2$ norm functions:
\begin{align}
f_{\mathrm{WLAV}}(\bnu)  & = \frac{|\nu_1|}{\sigma_1} + \frac{|\nu_2|}{\sigma_2} + \cdots + \frac{|\nu_M|}{\sigma_M}\\
f_{\mathrm{WLS}}(\bnu)  & = \frac{\nu_1^2}{\sigma_1^2} + \frac{\nu_2^2}{\sigma_2^2} + \cdots + \frac{\nu_M^2}{\sigma_M^2},
\end{align}
where $\sigma_1,...,\sigma_M$ are positive constants. 


\begin{remark}
The above functions correspond to the weighted least absolute value (WLAV) and weighted least square (WLS) estimators, which arise as the maximum likelihood estimator when the noises have a Laplace or normal distribution, respectively. Note that possible outliers in the measurements can be better modeled by the Laplace distribution that features heavier tails than the normal. Consequently, the WLAV estimator is more robust to the outliers. On the contrary, the non-robustness of the WLS estimator is primarily attributed to the squared distance because outliers with large residuals can have a high influence to skew the regression.
\end{remark}


Due to the inherent quadratic relationship between the voltage vector $\bv$ and the measured quantities $\{|v_i|^2,\bp,\bq,\bp_l,\bq_l\}$, the
quadratic equality constraints \eqref{PSSE2:constraint} make the problem \eqref{prob:PSSE2}
non-convex and NP-hard in general. To remedy this drawback, consider the penalized SDP relaxation
\begin{subequations}\label{PSSE-SDPP}
\begin{align}
\mini_{\bX \in \mathbb{H}^N, \bnu \in \mathbb{R}^M}\quad & \rho f(\bnu)+\Tr(\bM_0\bX) \\
\st\quad  &\Tr(\bM_j\bX)+\nu_j = z_j,\quad \forall j\in \cM \\
\quad & \bX \succeq \mathbf{0},
\end{align}
\end{subequations}
where $\rho>0$ is a pre-selected coefficient that balances the data fitting cost $f(\bnu)$ with the convexification term
$\Tr(\bM_0\bX)$.  The latter term is inherited from the SDP relaxation for the PF problem to deal with the non-convexity of the power flow equations. Similarly, a penalized  SOCP relaxation problem can be derived as
\begin{subequations}\label{PSSE-SOCP}
	\begin{align}
	\mini_{\bX \in \mathbb{H}^N, \bnu \in \mathbb{R}^M}\quad &\rho f(\bnu)+\Tr(\bM_0\bX) \label{PSSE-SOCP:obj} \\
	\st\quad  &\Tr(\bM_j\bX)+\nu_j = z_j,\, &&\forall\, j\in \cM \label{PSSE-SOCP:meq} \\
	\quad &
	\begin{bmatrix}
	X_{s,s} & X_{s,t}\\
	X_{t,s} &X_{t,t}
	\end{bmatrix}
	\succeq  \mathbf{0}, &&\forall~(s,t) \in \overline\cL,\label{PSSE-SOCP:cone}
	\end{align}
\end{subequations}
where $\overline \cL$ denotes the set of edges of the network graph for which
the corresponding entry of $\mathbf{M}_{j}$ is nonzero for at least one index $j \in \{0,1,\ldots,M\}$.
Based on the results derived earlier for  the PF problem,
we will next develop strong theoretical results on the estimation error for the PSSE.

\subsection{Bounded Estimation Error}
In this subsection, we assume that the function $f(\bnu)$ corresponds to the WLAV estimator, and that the available measurements consist of the voltage magnitudes at all buses and the active power flow at the ``from'' end of each line of $\cG^{\prime}$. The results to be presented next hold true in presence of extra  power measurements (see Remark~\ref{rem:r2}).
The penalized  problem~\eqref{PSSE-SDPP} can  be expressed as
\begin{align}\label{PSSE-SDPP2}
\min_{\bX \succeq \mathbf{0}}\, \Tr(\bM_0\bX)\!+\!\rho \sum_{j=1}^M \sigma_j^{-1}\left|\Tr\left(\bM_j(\bX\!-\!\bv\bv^{*})\right)\!-\!\eta_j\right|.
\end{align}
We aim to show that the solution of the penalized  relaxation  estimates the true solution of PSSE, where the estimation error is a function of the noise power. Define $\boldsymbol{\eta}$ as the vector of the noise values $\eta_1,..,\eta_M$.

%

\begin{theorem}\label{thm:rmse}
Suppose that  Assumption \ref{asmp1} holds. Consider an arbitrary dual SDP certificate $\hat{\boldsymbol{\mu}}\in\mathcal{D}(\mathbf{v})$, where $\bf v$ is the unique solution of the PSSE problem. Let $(\mathbf{X}^{\mathrm{opt}},\boldsymbol{\nu}^{\mathrm{opt}})$ denote an optimal solution of the penalized convex program \eqref{PSSE-SDPP} with $f(\bnu)=f_{\mathrm{WLAV}}(\bnu)$ and   a coefficient $\rho$ satisfying the inequality
\begin{align}\label{eq:rho}
\rho \geq \max_{j\in\cM} |\sigma_j\hat{\mu}_j|.
\end{align}  There exists a scalar $\beta >0$ such that
\begin{align} \label{rmse:Xopt}
\zeta := \frac{\|\mathbf{X}^{\mathrm{opt}}\! -\! \beta\mathbf{v}\mathbf{v}^{*}\|_F}{\sqrt{N\times \Tr({\bX}^{\mathrm{opt}})}}
\leq 2\sqrt{\frac{\rho \!\times\! f_{\mathrm{WLAV}}(\boldsymbol{\eta})}{N\lambda}},
\end{align}	
where $\lambda$ is the second smallest eigenvalue of the matrix $\mathbf{H}(\hat{\boldsymbol{\mu}})$.
\end{theorem}
\begin{proof}
The proof  is provided in  Appendix~\ref{appendix:rmse}.\end{proof}

Note that the numerator of $\zeta$ quantifies the distance between the optimal solution of the penalized convex program and the true PSSE solution.
The denominator of $\zeta$ is expected to be around $N$ since
$\Tr({\bX}^{\mathrm{opt}})\simeq N$ in the noiseless scenario.
Hence, the quantity $\zeta$ can be regarded as a root-mean-square estimation error.
Theorem~\ref{thm:rmse} establishes an upper bound for the estimation error as a function of the noise power
$f_{\mathrm{WLAV}}(\boldsymbol{\eta})$. In particular, the error is zero if $\boldsymbol{\eta}=0$.
This theorem provides an upper bound on the estimation error without using any statistical information of the random vector $\boldsymbol{\eta}$.
In what follows, the upper bound will be further studied for Gaussian random variables.
To this end, define $\kappa$ as $\frac{M}{N}$. If $M$ were the number of lines in the network, $\kappa$ was between 1.5 and 2 for most real-world power systems \cite{Chow13}.

\begin{corollary}\label{coro:probbound}
Suppose that the noise $\boldsymbol{\eta}$ is a zero-mean Gaussian vector with the covariance matrix $\bSigma = \diag(\sigma_1^2,...,\sigma_M^2)$.
Under the assumptions of Theorem~\ref{thm:rmse}, the tail probability of the estimation error $\zeta$ is upper bounded as
\begin{align}
\mathbb{P}(\zeta>t) \leq \myexp^{-\gamma M}
\end{align}
for every  $t>0$, where $\gamma = \frac{t^4\lambda^2}{32\kappa^2\rho^2}-\ln2$.
\end{corollary}
\begin{proof}
The proof is given in  Appendix \ref{appendix:probbound}.
\end{proof}

%

Recall that the measurements used for solving the PSSE problem include one active power flow per each line of the subgraph $\cG^{\prime}$. The graph $\cG^{\prime}$ could be as small as a spanning tree of $\mathcal G$ or as large as the entire graph $\mathcal G$. Although the results developed in this paper work in all of these cases, the number of measurements could significantly vary for different choices of $\cG^{\prime}$. A question arises as to how the number of measurements affects the estimation error. To address this problem, notice that if it is known that some measurements are corrupted with high values of noise, it would be preferable to discard those bad measurements.  To avoid this scenario,  assume that there are two sets of measurements with similar noise levels. It is aimed to show that the set with a higher cardinality would lead to a better estimation error.

\begin{definition} \label{def:dd1} Define $\omega (\cG^{\prime})$ as the minimum of $2\sqrt{\frac{\rho}{N\lambda}}$ over all  dual SDP certificates $\hat{\boldsymbol{\mu}}\in\mathcal{D}(\mathbf{v})$, where $\rho= \max_{j\in\cM} |\sigma_j\hat{\mu}_j|$ and $\lambda$ denotes the second smallest eigenvalue of $\mathbf{H}(\hat{\boldsymbol{\mu}})$.
\end{definition}

In light of Theorem~\ref{thm:rmse},
the  estimation error $\zeta$ satisfies the inequality
\begin{equation}
\zeta\leq \omega (\cG^{\prime}) \sqrt{f_{\mathrm{WLAV}}(\boldsymbol{\eta})}
\end{equation}
if an optimal coefficient $\rho$ is used in the penalized convex problem. The term $ \sqrt{f_{\mathrm{WLAV}}(\boldsymbol{\eta})}$ is related to the noise power. If this term is kept constant, then the estimation error is a function of $\omega (\cG^{\prime})$. Hence, it is desirable to analyze $\omega (\cG^{\prime})$.

\begin{theorem}\label{thm:measurment} Consider two choices of the graph $\cG^{\prime}$, denoted as $\mathcal G^{\prime}_1$ and $\mathcal G^{\prime}_2$, such that $\mathcal G^{\prime}_1$ is a subgraph of $\mathcal G^{\prime}_2$. Then, the relation
\begin{equation}
\omega (\mathcal G^{\prime}_2)\leq \omega (\mathcal G^{\prime}_1)
\end{equation}
holds.
\end{theorem}
\begin{proof}
The proof follows from the fact that the feasible set of the dual certificate $\hat{\boldsymbol{\mu}}$ for the case  $\mathcal G^{\prime}=\mathcal G^{\prime}_1$ is contained in the feasible set  of $\hat{\boldsymbol{\mu}}$ for $\mathcal G^{\prime}=\mathcal G^{\prime}_2$.
\end{proof}

The penalized  convex program \eqref{PSSE-SDPP}  may have a non-rank-1 solution in the noisy case.
Whenever the optimal solution  $\bX^{\text{opt}}$ is not rank 1, an estimated voltage vector $\hat{\bv}$ can be obtained using a rank-1 approximation method, such as the following  algorithm borrowed from \cite{madani2014promises}:
\begin{itemize}
\item [i)] Set the voltage magnitudes via the equations
\begin{align}\label{anglerecovery}
|\hat{v}_k| = \sqrt{\bX_{k,k}^{\text{opt}}}, \quad k=1,2,\ldots,N.
\end{align}
\item [ii)] Set the voltage angles via the convex program
\begin{subequations}\label{magrecovery}
\begin{align}
\measuredangle \hat{\bv} = &\argmin_{\measuredangle \bv \in [-\pi, \pi]^N} \sum_{(s,t)\in \cL}
|\measuredangle \bX_{s,t}^{\text{opt}}- \measuredangle v_{s} + \measuredangle v_{t}| \\
&\st \quad \measuredangle v_{\text{ref}} = 0.
\end{align}
\end{subequations}
\end{itemize}
Note that  $\hat{\bv}$ is the true solution of the PSSE problem if   $\bX^{\text{opt}}$ has rank 1.

\subsection{Reduction of Computational Complexity}

Due to the presence of the positive semidefinite constraint $\bX \succeq \mathbf{0}$,
 solving the  conic problems \eqref{PF-SDPP} and \eqref{PSSE-SDPP} is computationally expensive or
even prohibitive for large-scale power systems.
In this subsection, we deploy  a graph-theoretic approach  to replace
the complicating constraint $\bX \succeq \mathbf{0}$
with a set of small-sized SDP or SOCP constraints.

\begin{definition}
The sparsity graph of the problem \eqref{PF-SDPP} or \eqref{PSSE-SDPP} is defined as the
union of the sparsity graphs of the coefficient matrices $\mathbf{M}_j$ for $j = 0,1,\ldots,M$.
In other words, the sparsity graph of \eqref{PF-SDPP}  or \eqref{PSSE-SDPP} denoted as $\tilde{\mathcal G}=(\mathcal N,\tilde{\mathcal L})$
is a simple undirected graph with $N$ vertices,
which has an edge between every two  distinct vertices $s$ and $t$ if and only if
the $(s,t)$ entry of $\mathbf{M}_{j;st}$ is nonzero for some $j \in \{0,1,\ldots,M\}$.
\end{definition}

\begin{definition}[Tree decomposition]
A tree decomposition of $\tilde{\mathcal G}$ is a 2-tuple $(\cB, \cT)$,
where $\cB = \{\cB_1,\ldots,\cB_Q\}$ is a collection of subsets of $\cN$ and $\cT$ is a tree whose nodes (called \emph{bags}) are the subsets $\cB_r$ and satisfy the following properties:
\begin{itemize}
  \item Vertex coverage: Each vertex of $\tilde{\mathcal G}$ is a member of at least one node of $\cT$, i.e., $\cN = \cB_1 \cup \cdots \cup \cB_Q$.
  \item Edge coverage: For every edge $(s, t)$ in $\tilde{\mathcal G}$, there is a bag $\cB_r$ that contains both ends $s$ and $t$.
  \item Running intersection: For every two bags $\cB_i$ and $\cB_j$ in $\cT$, every node on the path connecting $\cB_i$ and $\cB_j$ contains $\cB_i \cap \cB_j$.
      In other words, all nodes of $\cT$ that contain a common vertex of $\tilde{\mathcal G}$ should form a subtree.
\end{itemize}
\end{definition}

\begin{theorem}\label{chor_them_1}
The optimal objective values of the SDP problems \eqref{PF-SDPP} and \eqref{PSSE-SDPP}  do not change if
their constraint $\bX \succeq \mathbf{0}$ is replaced by the set of constraints
\begin{align}\label{decomSDP}
  \bX[\cB_r,\cB_r] \succeq \mathbf{0}, \quad \forall r\in\{ 1,2,\ldots,Q\}.
\end{align}
\end{theorem}
\begin{proof}
This theorem is a direct consequence of the matrix completion theorem and chordal extension \cite{Grone1984}.
\end{proof}

As a by-product of Theorem~\ref{chor_them_1}, all off-diagonal entries of $\bX$ that do not appear in the submatrices $\bX[\cB_r,\cB_r]$ are redundant and could be eliminated from the SDP relaxations. This significantly reduces the computational complexity for sparse power systems. As an example, consider the case where  the sparsity graph  $\tilde{\mathcal G}$ is acyclic. Then,  $\tilde{\mathcal G}$ has a tree decomposition
such that each bag contains only two connected vertices of $\tilde{\mathcal G}$.
Hence, the decomposed constraints \eqref{decomSDP}
boil down to positive semidefinite constraints on a set of $2 \times 2$ submatrices of $\bX$.
This special case is formalized below.

\begin{corollary}
Suppose that the sparsity graph $\tilde{\mathcal G}$ is  a spanning tree of $\mathcal{G}$. Then,  the optimal objective value of the penalized SDP problem \eqref{PSSE-SDPP} is equal to the optimal objective value of the penalized SOCP  problem \eqref{PSSE-SOCP}.
\end{corollary}
It can be readily shown that the number of scalar optimization variables associated with the SOCP relaxation  \eqref{PSSE-SOCP} (after eliminating redundant variables) is $\mathcal{O}(N)$
as opposed to $\mathcal{O}(N^2)$ for the SDP relaxation \eqref{PSSE-SDPP}.

\section{Numerical Tests}\label{sec:test}

In this section, numerical results are presented to verify the performance of the proposed convexification techniques for
the PSSE problem. The tests are conducted on several benchmark power systems \cite{Josz16}, where the admittance matrices and the
underlying system states are obtained from \texttt{MATPOWER} \cite{matpower}.
Unless otherwise stated, the available measurements are assumed to be: (i) voltage magnitudes at all buses, and
(ii) one active power flow per line of a spanning tree of the network.
 The tree is obtained by the function \texttt{graphminspantree} in the Matlab bioinformatics toolbox \cite{matlabMST}.


We first compare the proposed SOCP relaxation \eqref{PSSE-SOCP} with the conventional WLS estimator (by using the Matpower function \texttt{run\_se} with flat start) for estimating the true complex voltage vector $\bv$. The performance metric is the root-mean-square error (RMSE) of the estimated voltage $\hat{\bv}$, which is defined as $\xi(\hat{\bv}):=\|\hat{\bv}-\bv\|_2/\sqrt{N}$.
The simulation results tested on the IEEE 57-bus and 118-bus systems are shown in Figures \ref{WLS_Newton}(a) and \ref{WLS_Newton}(b), respectively.
In each case, the measurements are under 100 randomly generated realizations of noise, which correspond to the voltage magnitudes for all buses and the active power flows at both ends of all lines. The zero-mean Gaussian noises have 0.002 and 0.001 per unit standard deviations for squared voltage magnitudes and line flows, respectively. In addition, $20\%$ of randomly chosen line flow measurements are generated as bad data,
which are contaminated by adding zero-mean Gaussian noises with 0.1 per unit standard deviation.
The coefficient matrix $\mathbf{M}_0$ is chosen as a real symmetric matrix with negative values at entries corresponding to the line flow measurements and zero elsewhere. The penalty weight is set to $\rho=1$ for all test cases.
Clearly, the penalized SOCP method significantly outperforms the conventional Newton-based WLS estimator.

Furthermore, we evaluate the effect of different types of measurements and scaling of load demand on the performance of PSSE. The simulation results are shown in Figure \ref{WLS_Newton_2}. In Figure \ref{WLS_Newton_2}(a), voltage measurements are only given at the reference and load (PQ) buses. 
In addition to the active power flows at both ends of all lines, reactive power flows are available at ``to'' ends of half of the lines.  Despite the fact that our assumption on voltage measurements does not hold in this case,
the proposed approach still has much smaller RMSEs. Similarly, performance gains are observed in Figure \ref{WLS_Newton_2}(b), where all fixed loads are scaled up 10\%.

\begin{figure}[t]
	\centering
	\hspace{-0.7cm}\subfloat[\label{WLS_Newton_57}]{\includegraphics[width = 7.7cm]{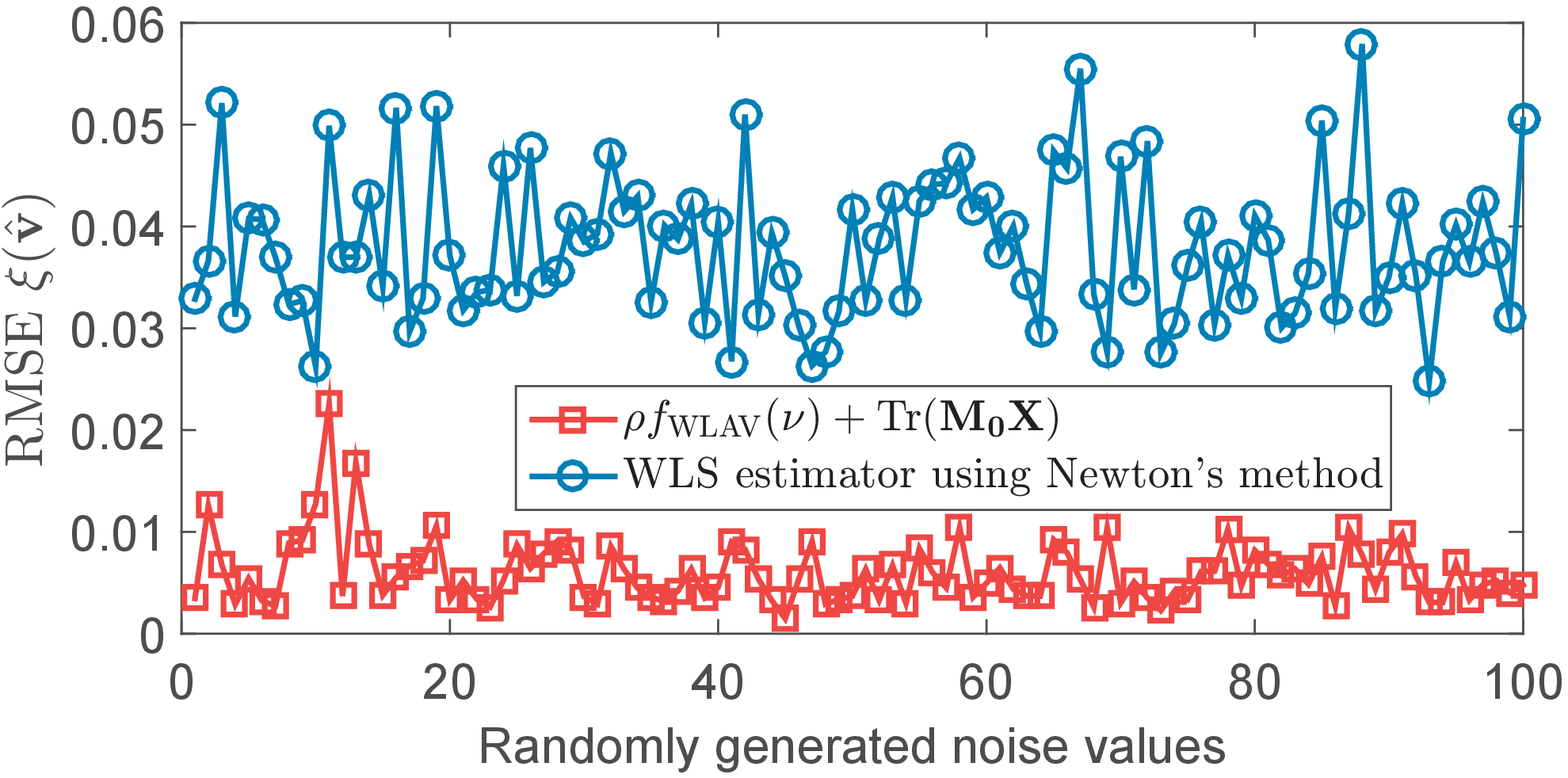}}\\
	\hspace{-0.7cm}\subfloat[\label{WLS_Newton_118}]{\includegraphics[width = 7.8cm]{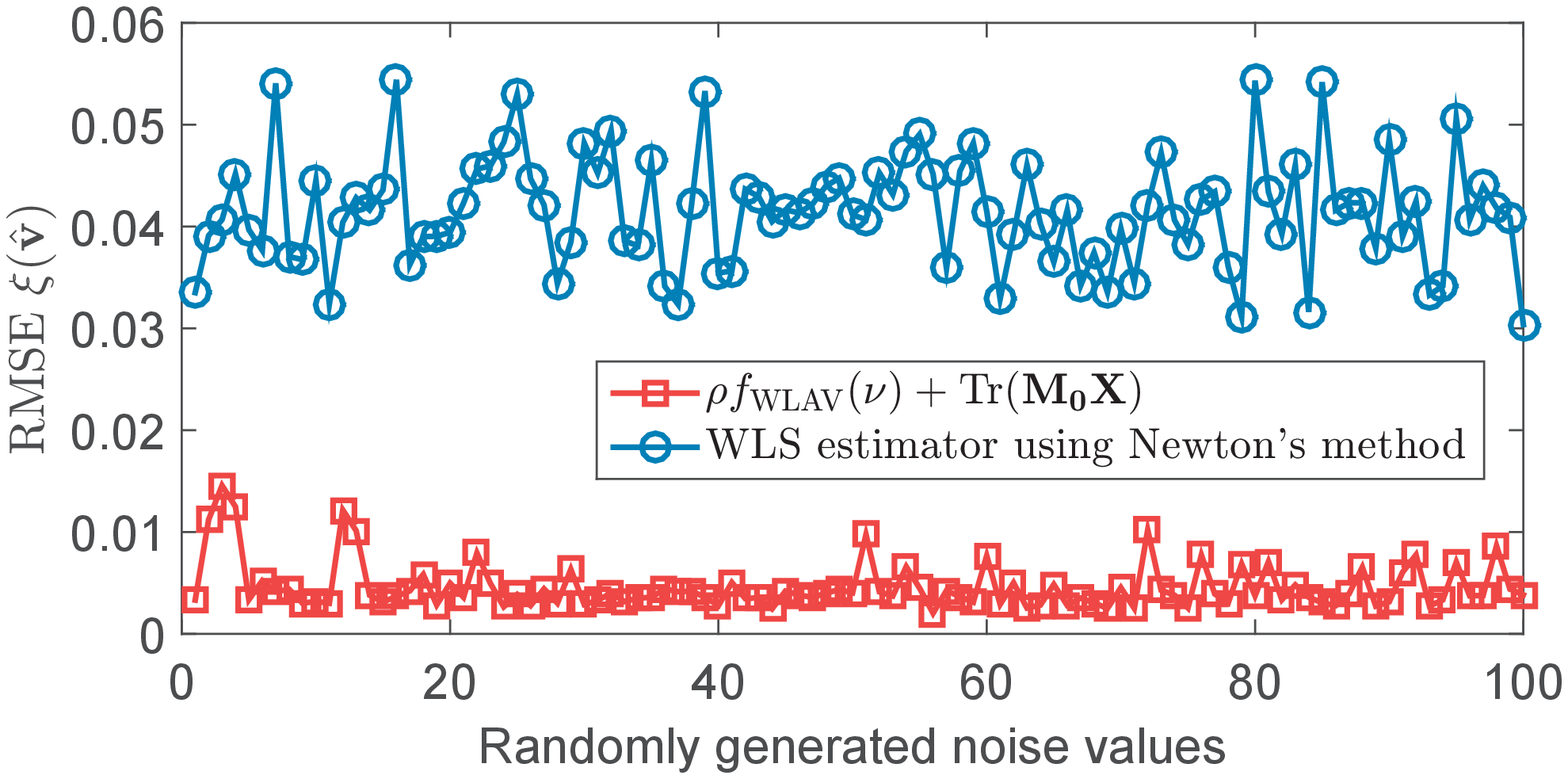}}
	\caption{The RMSEs of the estimated voltages obtained by the penalized SOCP method and the WLS-based Newton method: (a) IEEE 57-bus system, (b) IEEE 118-bus system.}\label{WLS_Newton}
\end{figure}

\begin{figure}[t]
	\centering
	\hspace{-0.7cm}\subfloat[\label{WLS_Newton_partV}]{\includegraphics[width = 7.7cm]{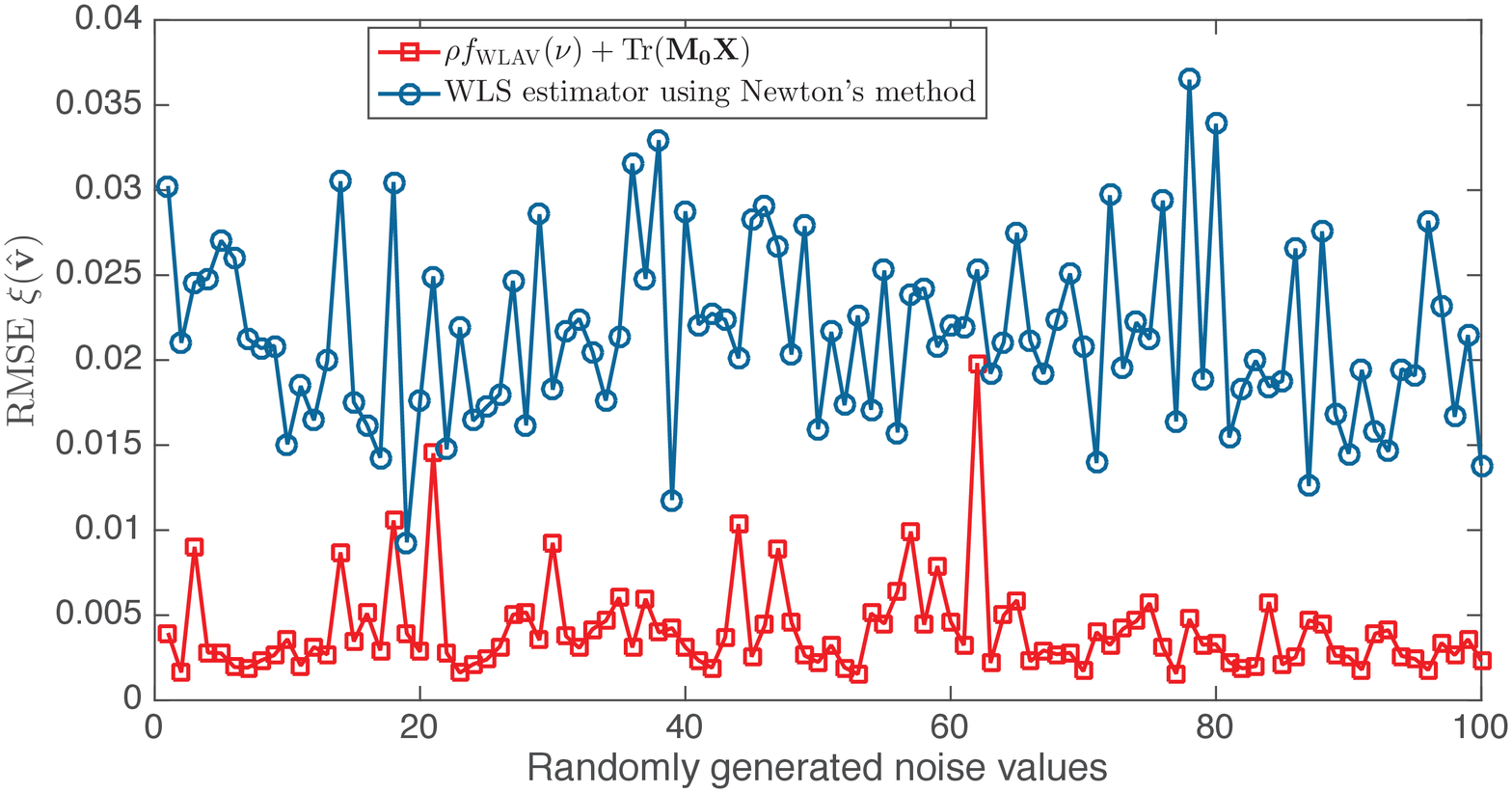}}\\
	\hspace{-0.7cm}\subfloat[\label{WLS_Newton_scaleload}]{\includegraphics[width = 7.7cm]{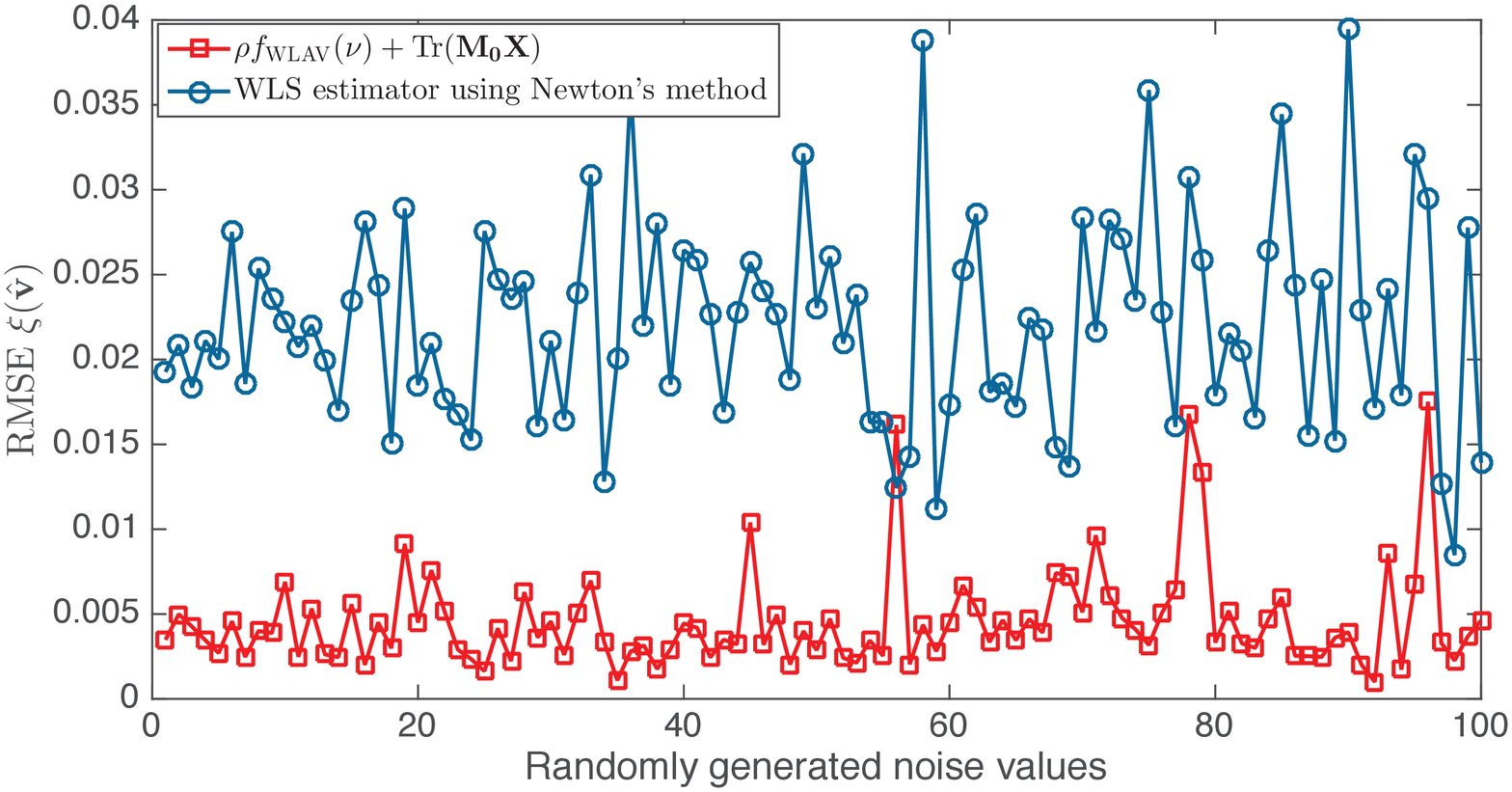}}
	\caption{The RMSEs of the estimated voltages obtained by the penalized SOCP method and the WLS-based  Newton method for the IEEE 57-bus system: (a) voltage magnitude measurements are not available at generator (PV) buses, (b) active and reactive power of all loads are scaled up 10\%.}\label{WLS_Newton_2}
\end{figure}

The numerical results for the penalized SDP relaxation problem \eqref{PSSE-SDPP} performed on several benchmark systems are shown in Tables \ref{tab:perf1} and \ref{tab:perf2}. The following numbers shown in \eqref{rmse:Xopt} are reported for each case:
\begin{itemize}
\item $\zeta$: the RMSE of the obtained optimal SDP solution $\bX^{\text{opt}}$. 
\item $\zeta^{\text{max}}$: the upper bound of $\zeta$.
\item Other relevant quantities $\beta$, $\lambda$, $f_{\text{WLAV}}$ and $\rho^{\text{min}}$.
\end{itemize}
In this test, for each squared voltage magnitude $\{|v_k|^2\}_{k\in \cN}$, the standard deviation of the zero-mean Gaussian noise is chosen $c$ times higher than its noiseless value, where $c>0$ is a pre-selected scalar quantifying the noise level. Likewise, the standard deviations for nodal and branch active/reactive power measurements are $1.5c$ and $2c$ times higher than the corresponding noiseless values, respectively. 
The entries of matrix $\mathbf{M}_0$ are set as  $\mathbf{M}_{0;st} = -\mathbf{B}_{st}$ for all
$(s,t)\in\cL^{\prime}$, and  $\mathbf{M}_{0;ii} = \sum_{j=1}^N|\mathbf{B}_{i,j}|$ for $i=1,2,\ldots,N$.
The penalty weight is set to $\rho^{\text{min}}:= \max_{j\in\cM} |\sigma_j\hat{\mu}_j|$ as given in \eqref{eq:rho}.

\begin{table}[t]
\centering
\caption{Performance of the penalized SDP \eqref{PSSE-SDPP} with the noise level $c=0.01$.}\label{tab:perf1}
\begin{tabular}{|p{7.4mm}|p{6.5mm}|p{6.5mm}|p{6.5mm}|p{6.5mm}|p{6.5mm}|p{7.5mm}|p{6.5mm}|}
\hline
\text{Cases}     &$\xi(\hat{\bv})$     &$\zeta$	   &$\zeta^{\text{max}}$	 &$\beta$	&$\lambda$	&$f_{\text{WLAV}}$    &$\rho^{\text{min}}$ \\  \hline
\text{9-bus}	  &  0.0111 &     0.0145 &     0.1535 &     0.9972 &     1.3417 &    14.768 &     0.0048         \\
\text{14-bus}	  & 0.0057  &    0.0078  &    0.2859  &    1.0005  &    0.3812  &   20.509  &    0.0053       \\
\text{30-bus}	  & 0.0060  &    0.0084  &    0.3728  &    0.9997  &    0.1094  &   51.479  &    0.0022        \\
\text{39-bus}	  & 0.0077  &    0.0083  &    0.8397  &    1.0009  &    0.7438  &   62.558  &    0.0817          \\
\text{57-bus}	  & 0.0092  &    0.0102  &    0.8364  &    1.0013  &    0.0912  &   88.434  &    0.0103           \\
\text{118-bus}	  & 0.0057  &    0.0079  &    1.2585  &    0.9992  &    0.0878  &  179.509  &    0.0228           \\
\hline
\end{tabular}
\end{table}

\begin{table}[t]
\centering
\caption{Performance of the penalized SDP \eqref{PSSE-SDPP} with the noise level $c=0.1$.}\label{tab:perf2}
\begin{tabular}{|p{7.4mm}|p{6.5mm}|p{6.5mm}|p{6.5mm}|p{6.5mm}|p{6.5mm}|p{7.5mm}|p{6.5mm}|}
\hline
\text{Cases}     &$\xi(\hat{\bv})$     &$\zeta$	   &$\zeta^{\text{max}}$	 &$\beta$	&$\lambda$	&$f_{\text{WLAV}}$    &$\rho^{\text{min}}$ \\  \hline
\text{9-bus}	  &  0.0357     &    0.0462   &     0.4237  &     0.9779 &     1.3417 &      11.250  &        0.0482        \\
\text{14-bus}	  &  0.0418     &    0.0537   &     0.8119  &     0.9682 &     0.3812 &      16.536  &       0.0532           \\
\text{30-bus}	  &   0.0297    &     0.0405  &     1.1734  &     0.9882 &     0.1094 &      50.993  &       0.0222           \\
\text{39-bus}	  &   0.0485    &     0.0676  &     2.4315  &     0.9840 &     0.7438 &      52.462  &       0.8173           \\
\text{57-bus}	  &   0.0907    &     0.1028  &     2.6937  &     1.0393 &     0.0912 &      91.724  &       0.1028               \\
\text{118-bus}	  &   0.0559    &     0.0743  &     4.0302  &     0.9871 &    0.0878  &      184.093  &        0.2284         \\
\hline
\end{tabular}
\end{table}

\begin{table}[t]
\centering
\caption{The average RMSEs of the estimated voltage vector $\hat{\bv}$ obtained by
the penalized SDP \eqref{PSSE-SDPP} for six different objective functions with the noise level $c=0.1$.}\label{tab:4obj}
\begin{tabular}{|p{8.2mm}|p{6.5mm} p{6.5mm}|p{6.5mm} p{6.5mm}|p{6.5mm} p{6.5mm}| } 
\cline{1-7}
\multirow{2}{*}{Methods}  & \multicolumn{2}{c|}{$\rho f(\bnu)+\Tr(\bM_0\bX)$} &  \multicolumn{2}{c|}{ $\rho f(\bnu)+\nuclearnorm{\bX}$} &  \multicolumn{2}{c|}{$ \rho f(\bnu)$} \\
\cline{2-7}
& \text{WLAV} &  \text{WLS} & \text{WLAV} &  \text{WLS} & \text{WLAV} &  \text{WLS} \\
\cline{1-7}
\text{9-bus}	   &  0.0648      & 0.1293    &    1.2744      & 1.1483       &    1.1619      &  1.1633      \\
\text{14-bus}      &  0.1307      &  0.1784   &  1.1320     &    1.3871       &    1.4233      &   1.4215          \\
\text{30-bus}      &  0.2055      &  0.2543   &   1.4236      &   1.4306      &    1.4269      &  1.4268            \\
\text{39-bus}      &  0.1324      &  0.1239   &  1.1317      &  1.3135       &    1.2764      &   1.2757            \\
\text{57-bus}      &  0.2343      & 0.2809    &  1.2981     & 1.3004        &    1.3235      &    1.3098          \\
\text{118-bus}     &  0.1136      & 0.1641    &  1.3620      &  1.3272       &    1.3445      &     1.3577         \\
\hline
\end{tabular}
\end{table}

\begin{table}[t]
\centering
 \caption{Simulation times of the penalized conic relaxations with 
 $f_{\text{WLAV}}(\boldsymbol{\nu})$ and the noise level $c=0.1$ (the unit is second).}\label{tab:simuTime}
\begin{tabular}{|l|c|c|}
\hline
\text{Cases}     &\text{Solver time}     &\text{Total time}   \\  \hline
\text{9-bus}	  &  0.89 &     1.58   \\
\text{14-bus}	  & 1.23  &    2.54       \\
\text{30-bus}	  & 1.33  &    3.21         \\
\text{39-bus}	  & 1.56  &    3.28           \\
\text{57-bus}	  & 1.97  &    4.09            \\
\text{118-bus}	  & 2.38  &    5.63             \\
\text{1354-bus}	  & 4.55  &    9.48           \\
\text{2869-bus}	  & 13.17  &    24.44            \\
\text{9241-bus}	  & 58.00  &    109.14             \\
\hline
\end{tabular}
\end{table}

For all test cases, it can be observed that the obtained optimal solutions of the penalized SDP method yield good estimates of the complex voltages featuring small RMSEs $\xi(\hat{\bv})$ and $\zeta$. These two error metrics are roughly on the same order as the corresponding noise levels.
Furthermore, the value of $\zeta^{\text{max}}$ is calculated using the quantities $\rho$ and $\lambda$. As expected, this is a legitimate upper bound on $\zeta$ that corroborates our theoretical results in Theorem \ref{thm:rmse}. The tightness of this upper bound depends on the second smallest eigenvalue of the dual matrix $\bH(\hat{\bmu})$, which is a function of the
true state $\bv$ and the matrix $\bM_0$. The discrepancy between $\zeta$ and $\zeta^{\text{max}}$ is rooted in the fact that $\zeta$ corresponds to our realization of noise, but $\zeta^{\text{max}}$ works for all realizations of the noise independent of its statistical properties.
Moreover, the value of the scaling factor $\beta$ (see \eqref{rmse:Xopt}) is always very close to 1 for all scenarios.
This implies   that the optimal SDP solution $\bX^{\text{opt}}$
is close to the true lifted state $\bv\bv^{*}$ without scaling this rank-one matrix.

To further  show the merit of the  proposed penalized SDP framework,
we compare the performance of the convex problem \eqref{PSSE-SDPP} against two  other estimation techniques.
To this end, consider three convex programs that are obtained from  \eqref{PSSE-SDPP} by changing its objective to:
(i) $\rho f(\bnu)+\Tr(\bM_0\bX)$, (ii) $\rho f(\bnu)+\nuclearnorm{\bX}$ (see \cite{Weng15} and \cite{Kim15}),
 and  (iii)  $\rho f(\bnu)$ (see  \cite{Zhu11,Zhu14,Weng12,Weng13}).
Each of these methods is tested for  both WLAV and  WLS  functions. 
Furthermore,  $10\%$ of the measurements are generated as bad data to show the robustness of  WLAV compared with  WLS.
These bad data are simulated by adding uniformly distributed random numbers (over the interval $[0,2]$) to the original measurements.
Table \ref{tab:4obj} reports the RMSE $\xi(\hat{\bv})$ averaged over 50 Monte-Carlo simulations for each test case,
where the parameter $\rho$ is set to $0.1$.
The penalized SDP method proposed in this work clearly outperforms the other techniques.

To show the scalability of the proposed approaches, we conduct simulations on large-scale systems by solving the penalized SDP or SOCP relaxations.
Figure \ref{fig_2}  shows the effect of additional measurements on reducing the estimation error.
In Figures \ref{fig_2}(a) and \ref{fig_2}(b), the RMSEs of the
estimated voltage vectors $\hat{\bv}$ are depicted  for  four different objective functions 
with respect to the percentage of nodes having measured active power injections.
The measurements are under two samples of the noise $\boldsymbol{\eta}$ corresponding to $c=0.01$ and $c=0.02$. 
It can be observed that the quality of  the estimation improves with the increase of nodal active power  measurements.
Even in the case when the number of measurements is limited and close to the number of unknown parameters, the proposed approach can still produce good estimates. In contrast, the methods with no penalty yield very high errors that are out of the plot ranges.


\begin{figure*}
	\centering
	\subfloat[\label{fig_2a}]{ \includegraphics[width =0.33\textwidth]{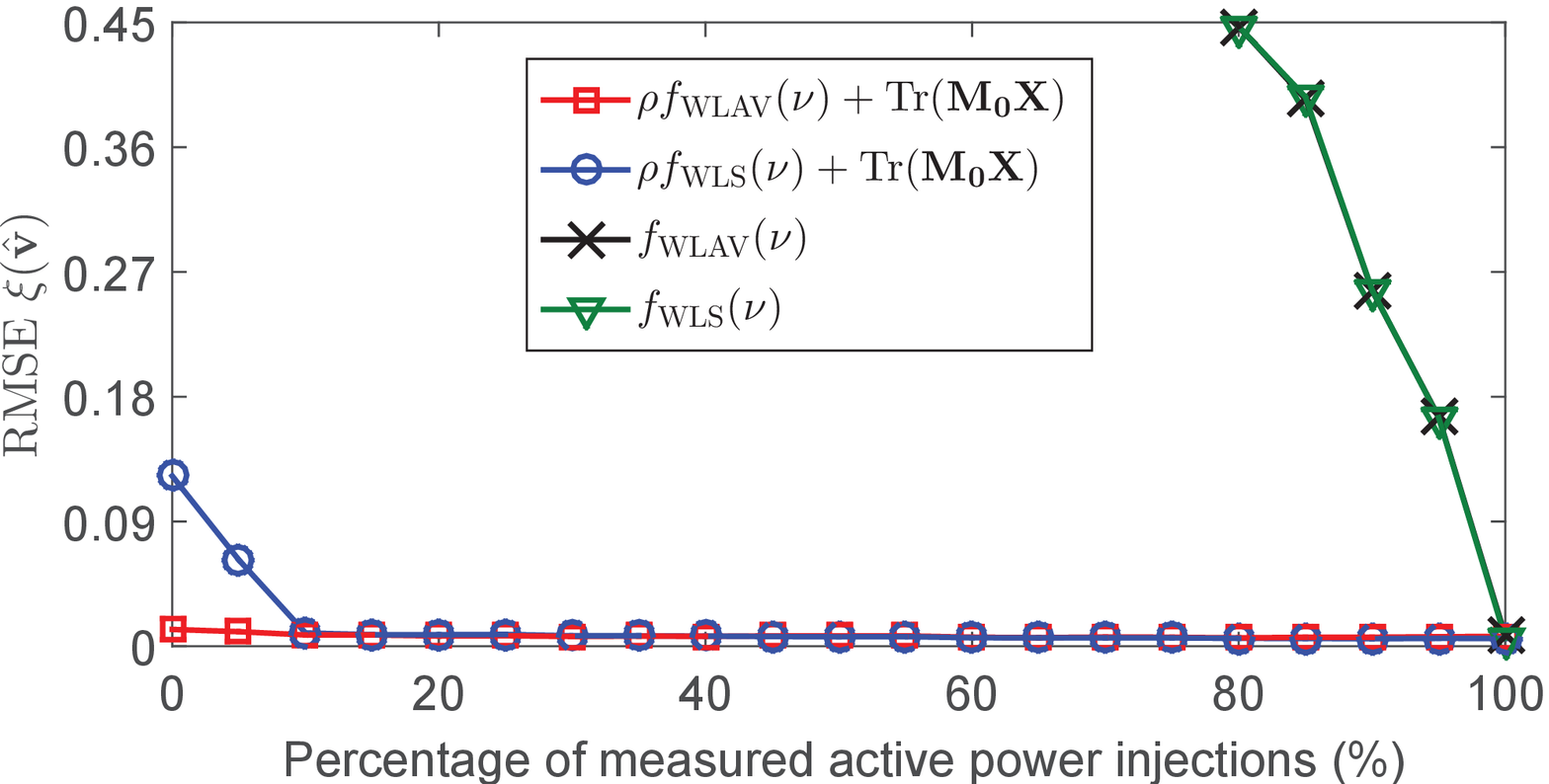}}
	\subfloat[\label{fig_2b}]{ \includegraphics[width =0.33\textwidth]{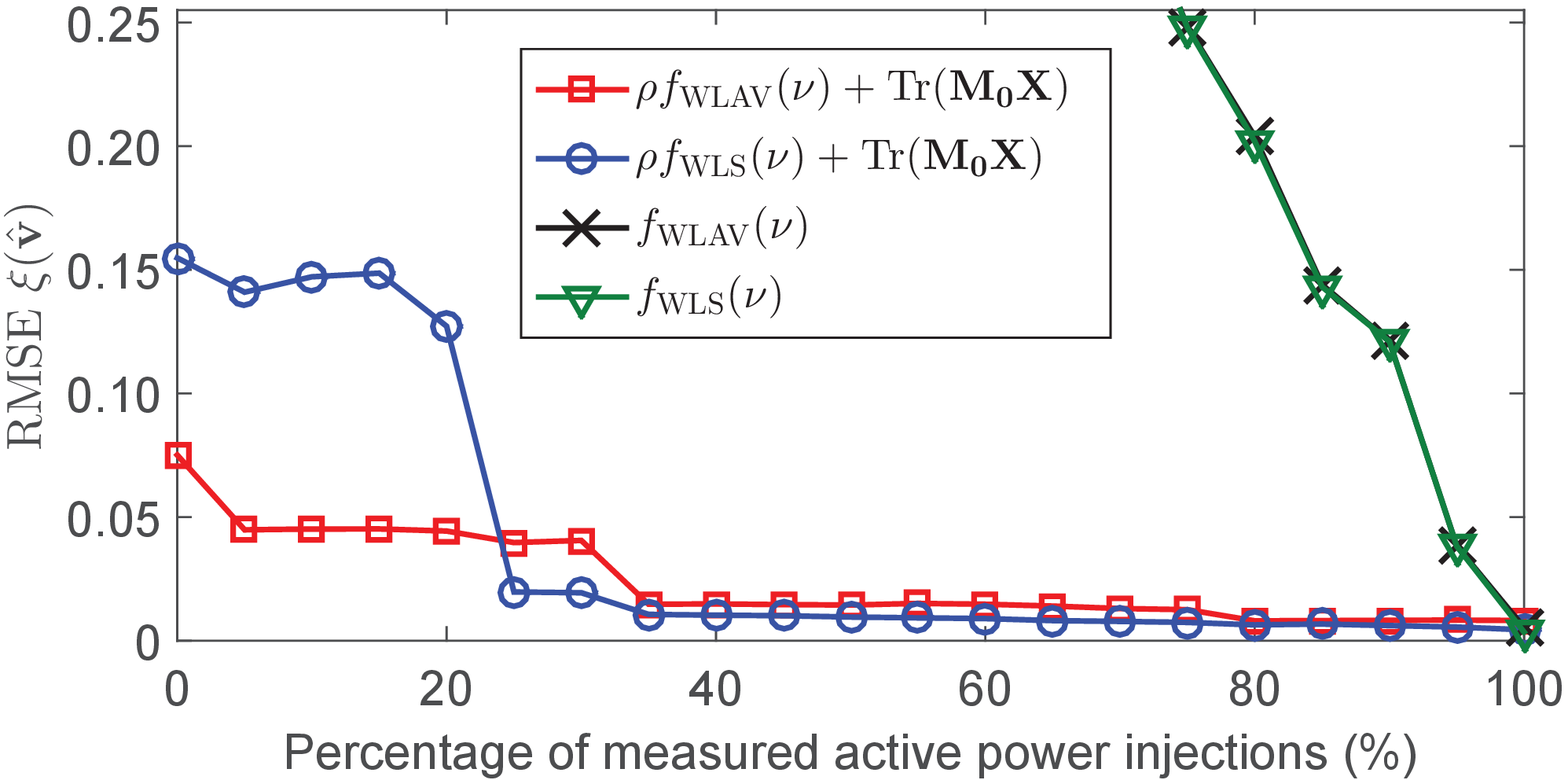}}
	\subfloat[\label{fig_2c}]{\includegraphics[width =0.33\textwidth]{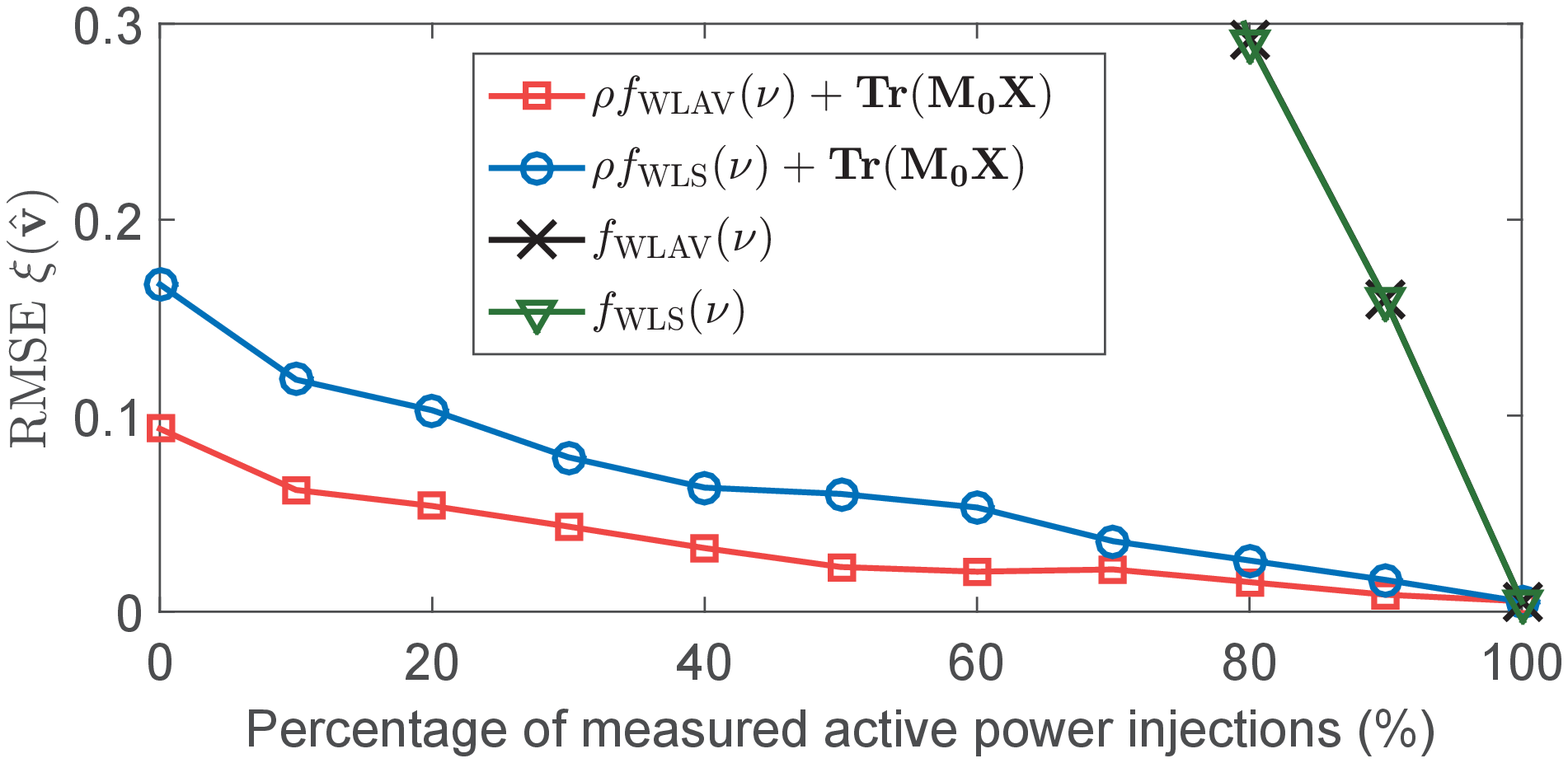}}
	\caption{The RMSEs of the estimated voltages obtained with  additional nodal power measurements: (a) $c=0.01$ and (b) $c=0.02$. Both are tested on PEGASE 1354-bus system using the penalized SDP;  (c) $c=0.01$ with PEGASE 9241-bus system using the penalized SOCP.}
	\label{fig_2}\vspace{-5mm}
\end{figure*}

In Figure \ref{fig_2}(c) for all four curves, it is assumed that the voltage magnitudes at all buses and active power flows in one direction for all branches are measured. Moreover, different percentages of nodes are chosen at which nodal active and reactive power measurements are made simultaneously.
The noise level is set to $c=0.01$ and the weight is $\rho = 5$. It can be observed that the quality of the estimation improves by increasing the number of additional measurements. The RMSE value at each data point is the average over 10 Monte-Carlo simulations for different noise realizations and choices of nodes with measured power injections.

Finally, Table \ref{tab:simuTime} lists the simulation time of the proposed conic relaxations. The total time is obtained by the command \texttt{cvx\_cputime}, which includes both \texttt{CVX} modeling time and solver time \cite{cvx}. For all benchmark systems from 9-bus to 118-bus, the SDP relaxation problem is solved by \texttt{SDPT3 4.0} \cite{sdpt3}. 
The simulation time is obtained by averaging over 50 Monte-Carlo simulations, which are tested on a macOS system with 2.7GHz Intel Core i5 and 8GB memory. For the last three large-scale test cases,  we utilize the SOCP relaxation with the solver~\texttt{MOSEK 7.0} \cite{mosek}. The simulation time corresponds to a single run, which is tested on a Windows system with 2.20GHz CPU and 12GB RAM. Clearly, it only takes a few seconds for each case (except the last one) to yield an optimal solution. Even for the large-scale 9241-bus network, the solver time for the proposed SOCP is less than 1 minute, which is fairly practical in real-world applications.

\section{Conclusions}\label{sec:Conclusions}

In this paper, a convex optimization framework is developed for solving the non-convex PF and PSSE problems.
To efficiently solve these two problems, the quadratic power flow equations are lifted into a 
higher-dimensional space, which enables their formulation as linear functions of a rank-one positive semidefinite matrix variable.  By meticulously designing an objective function, the PF feasibility problem is converted into a non-convex optimization problem and then relaxed to a convex program. The performance of the proposed convexification is studied in the case where the set of measurements includes:  (i) nodal voltage magnitudes, and (ii) one active power flow per line for a spanning tree of the power network. It is shown that the designed convex problem finds the correct solution of the PF problem as long as the voltage angle differences across the lines of the network are not too large.
This result along with the proposed framework is then extended to the PSSE problem.
Aside from the well-designed objective function for dealing with the non-convexity of PF,
a data fitting penalty based on the weighted least absolute value is included
to account for the noisy  measurements. This leads to a penalized conic optimization scheme.
The distance between the optimal solution of the proposed convex problem and the unknown state of the system is quantified in terms of the noise level,
which decays as the number of measurements increases. 
Extensive numerical results tested on benchmark systems corroborate our theoretical analysis.
Moreover, compared with the conventional WLS-based Newton's method as well as other convex programs with different regularizers, the proposed approaches have significant performance gains in terms of the RMSE of the estimated voltages.

\appendix

\subsection{Proof of Lemma~\ref{lem:dualcertf}}\label{appendix:dualcertf}
Let $\mu_1,\ldots,\mu_N\in\mathbb{R}$ and $\mu_{N+1},\ldots,\mu_{M}\in\mathbb{R}$ be the Lagrange multipliers associated with the constraints
\eqref{PF-SDPP-specM:meq_Node} and \eqref{PF-SDPP-specM:meq_Branch}, respectively.
A key observation is that each matrix $\bY_{l,pf}$ has only three possible nonzero entries:
\begin{align*}
\bY_{l,pf}(s,t) = \bY_{l,pf}^{\ast}(t,s) =  -\frac{y_{st}}{2},\quad \bY_{l,pf}(s,s) = \re(y_{st}).
\end{align*}
Hence, in order to design  $\hat{\boldsymbol{\mu}}\in\mathcal{D}(\mathbf{v})$, we first construct a vector $\hat{\boldsymbol{\mu}}^{(l)}$ for each $l\in\cL^{\prime}$ and then  show that the summation
\begin{align}
\hat{\boldsymbol{\mu}}:= \hat{\boldsymbol{\mu}}^{(1)}+\hat{\boldsymbol{\mu}}^{(2)}+\cdots+\hat{\boldsymbol{\mu}}^{(|\cL^{\prime}|)}
\end{align}
satisfies the requirements for being a dual SDP certificate.
Let $\{\tilde{\mathbf{e}}_1,\ldots,\tilde{\mathbf{e}}_{M}\}$ denote the canonical vectors in $\mathbb{R}^{M}$.
Under Assumption \ref{asmp1}, it is possible to set
\begin{align}
\hat{\boldsymbol{\mu}}^{(l)}\!:=\hat{\mu}^{(l)}_s \tilde{\mathbf{e}}_s + \hat{\mu}^{(l)}_t\tilde{\mathbf{e}}_t + \hat{\mu}^{(l)}_{N+l}\tilde{\mathbf{e}}_{N+l}
\end{align}
for each line $l$ that connects node $s$ to node $t$, such that
\begin{align}
\hat{\mu}^{(l)}_{N+l} &:= \frac{2\mathrm{Im}(v_s v^{\ast}_t M^{\ast}_{0;st})}{\mathrm{Im}(v_s v^{\ast}_t y^{\ast}_{st})},\,\,
\hat{\mu}^{(l)}_t := -\frac{|v_s|^2\mathrm{Im}(M_{0;st}y^{\ast}_{st})}{\mathrm{Im}(v_s v^{\ast}_t y^{\ast}_{st})}, \notag \\
\hat{\mu}^{(l)}_s &:= \frac{|v_t|^2}{|v_s|^2}\hat{\mu}^{(l)}_t-\mathrm{Re}(y_{st})\hat{\mu}^{(l)}_{N+l}.
\label{eq:mul}
\end{align}
In order to prove that $\hat{\boldsymbol{\mu}}$ satisfies the conditions in \eqref{dual_cer}, define
\begin{align}
\widehat{\mathbf{H}}^{(l)}:=\mathbf{M}^{(l)}_0+\hat{\mu}^{(l)}_s\mathbf{E}_s+\hat{\mu}^{(l)}_t\mathbf{E}_t+\hat{\mu}^{(l)}_{N+l}\mathbf{Y}_{l,p_f}
\end{align}
for every $l=(s,t)\in\cL_s$, where
\begin{align}
\mathbf{M}^{(l)}_0:= M_{0;st} {\mathbf{e}}_s{\mathbf{e}}^{\top}_t + M_{0;st}^{*} {\mathbf{e}}_t{\mathbf{e}}^{\top}_s.
\end{align}
It is easy to verify that the relations
\begin{align}
\widehat{\mathbf{H}}^{(l)}\mathbf{v}= \mathbf{0},\qquad \widehat{\mathbf{H}}^{(l)}\succeq \mathbf{0}
\label{eq:vnull}
\end{align}
hold for every $l\in\cL^{\prime}$.
Therefore, the matrix
\begin{align}
\mathbf{H}(\hat{\boldsymbol{\mu}})= \sum_{l\in\cL^{\prime}} \widehat{\mathbf{H}}^{(l)}
\end{align}
 satisfies the first two conditions in \eqref{dual_cer}.

It remains to prove that $\rank(\mathbf{H}(\hat{\bm{\mu}})) = N-1$, or equivalently
$\mathrm{dim}(\mathrm{null}(\mathbf{H}(\hat{\bm{\mu}}))) = 1$. Let $\bx$ be an arbitrary member of $\mathrm{null}(\mathbf{H}(\hat{\bm{\mu}}))$. Since each matrix $\widehat{\mathbf{H}}^{(l)}$ is positive semidefinite, we have $\widehat{\mathbf{H}}^{(l)}\bx = \mathbf{0}$ for every $l\in\mathcal{L}^{\prime}$.
As a result, it follows from \eqref{eq:mul} and \eqref{eq:vnull} that $\frac{x_s}{x_t} = \frac{v_s}{v_t}$
for $l=(s,t)$.
By the same reasoning, the relation $\frac{x_t}{x_a} = \frac{v_t}{v_a}$ holds true for each line $l^{\prime} = (t,a)$.
Upon defining $c:=\frac{x_s}{v_s}$, one can write
\begin{align*}
\frac{x_s}{v_s} = \frac{x_t}{v_t} = \frac{x_a}{v_a} = c.
\end{align*}
Repeating the above argument over the connected spanning subgraph ${\mathcal G}^{\prime}$ through all nodes $k \in \cN$ yields that
$\bx = c\bv$, and subsequently $\rank(\mathbf{H}(\hat{\bm{\mu}})) = N-1$.

\subsection{Proof of Theorem~\ref{thm:tightrelax}}\label{appendix:tightrelax}
By choosing sufficiently large values for the Lagrange multipliers associated with the voltage magnitude measurements \eqref{PF-SDPP-specM:meq_Node}, a strictly feasible point can be obtained for the dual problem \eqref{PF-SDPD}.  Therefore, strong duality holds between the primal and dual SDP problems.
According to Lemma \ref{lem:dualcertf}, there exists a dual SDP certificate $\hat{\boldsymbol{\mu}}\in\mathcal{D}(\mathbf{v})$ that satisfies  \eqref{dual_cer}. Therefore,
\begin{align}\label{2propty}
\mathrm{Tr}(\bH(\hat{\boldsymbol{\mu}})\mathbf{v}\mathbf{v}^{\ast}) = \mathbf{0},\qquad \bH(\hat{\boldsymbol{\mu}})\succeq \mathbf{0}.
\end{align}
This certifies the optimality of the point $\bX=\mathbf{v}\mathbf{v}^{\ast}$ for the SDP relaxation problem \eqref{PF-SDPP-specM}.
Moreover, the property $\rank(\bH(\hat{\boldsymbol{\mu}})) = N-1$
justifies the uniqueness of the primal SDP solution $\bX^{\text{opt}}$.

In order to prove recovery through the SOCP relaxation problem \eqref{PF-SOCP-specM}, it is useful to derive the dual SOCP problem:
\begin{subequations}
	\begin{align}
	\hspace{-2mm}\maxi_{
	\begin{subarray}{c}
	\bmu \in \mathbb{R}^M\\
	\mathbf{F}^{(1)}, \ldots,
	\mathbf{F}^{|\mathcal{L}^{\prime}|} \in \mathbb{H}^2\\
	\end{subarray}
	} \hspace{-2mm}& \quad -\bz^{\top}\bmu \\
	\st  &
	\sum_{
		l=(s,t)\in\mathcal{L}^{\prime}
		} {
		\begin{bmatrix}
		\tilde{\mathbf{e}}_{s}^{\top}\\
		\tilde{\mathbf{e}}_{t}^{\top}
		\end{bmatrix}^{\top}
		\!\!\!\mathbf{F}^{(l)}
		\begin{bmatrix}
		\tilde{\mathbf{e}}_{s}^{\top}\\
		\tilde{\mathbf{e}}_{t}^{\top}
		\end{bmatrix}}=\bM_0 + \sum_{j = 1}^M \mu_j\bM_j , \\
	\quad &
\mathbf{F}^{(l)}
	\succeq  \mathbf{0},\quad\quad\forall~l \in \cL^{\prime},
	\end{align}
\end{subequations}
where $\mu_1,\ldots,\mu_N\in\mathbb{R}$ and $\mu_{N+1},\ldots,\mu_{M}\in\mathbb{R}$ are the Lagrange multipliers associated with the constraints
\eqref{PF-SOCP-specM:meq_Node} and \eqref{PF-SOCP-specM:meq_Branch} respectively, and each $2\times2$ matrix $\mathbf{F}^{(l)}$
is the Lagrange multiplier associated with its corresponding second-order cone constraint in \eqref{PF-SOCP-specM:cone}.
Let $\bmu = \hat{\bmu}$, and set $\mathbf{F}^{(l)}:=\widehat{\mathbf{H}}^{(l)}[(s,t),(s,t)]$ for every $l\in\mathcal{L}^{\prime}$.
As before, the primal and dual feasibility conditions are satisfied, and in addition the complementary slackness
\begin{align}
\Tr\left(\mathbf{F}^{(l)}
\begin{bmatrix}
|v_s|^2 & v_s v^{\ast}_t\\
v_t v^{\ast}_s & |v_t|^2
\end{bmatrix}
\right)= 0,
\quad\forall~(s,t) \in \cL^{\prime}
\end{align}
 holds between the dual feasible point
$(\bmu,\{\mathbf{F}^{(l)}\}^{M}_{l=1})$
and the primal feasible point $\mathbf{v}\mathbf{v}^{\ast}$.
Let $\mathbf{X}^{\mathrm{opt}}$ be an arbitrary solution of \eqref{PF-SOCP-specM}. One can write:
\begin{align}
\Tr\left(\mathbf{F}^{(l)}
\begin{bmatrix}
X^{\mathrm{opt}}_{s,s} & X^{\mathrm{opt}}_{s,t}\\
X^{\mathrm{opt}}_{t,s} & X^{\mathrm{opt}}_{t,t}
\end{bmatrix}
\right)= 0,
\quad\forall~(s,t) \in \cL^{\prime}.
\end{align}
Both matrices $\mathbf{F}^{(l)}$ and
$\begin{bmatrix}
	X^{\mathrm{opt}}_{s,s} & X^{\mathrm{opt}}_{s,t}\\
	X^{\mathrm{opt}}_{t,s} & X^{\mathrm{opt}}_{t,t}
\end{bmatrix}$
are positive semidefinite and besides
\begin{align}
\mathrm{null}(\mathbf{F}^{(l)})=\{c\times[v_s \; v_t]^{\top} \;|\; c\in\mathbb{R}\},
\end{align}
where $l=(s,t)$. As a result,
\begin{align}
\begin{bmatrix}
X^{\mathrm{opt}}_{s,s} & X^{\mathrm{opt}}_{s,t}\\
X^{\mathrm{opt}}_{t,s} & X^{\mathrm{opt}}_{t,t}
\end{bmatrix}=c^{\prime}\times
\begin{bmatrix}
|v_s|^2 & v_s v^{\ast}_t\\
v_t v^{\ast}_s & |v_t|^2
\end{bmatrix}
\end{align}
holds for some $c^{\prime}>0$. On the other hand, according to primal feasibility, we have $X^{\mathrm{opt}}_{s,s}=|v_s|^2$ and $X^{\mathrm{opt}}_{t,t}=|v_t|^2$, which means that $c^{\prime}=1$. Therefore, all diagonal and those off-diagonal entries of $\mathbf{X}^{\mathrm{opt}}$ associated with the members of $\mathcal{L}^{\prime}$ are equal to their corresponding entries in $\mathbf{v}\mathbf{v}^{\ast}$.

\subsection{Proof of Theorem~\ref{thm:rmse}}\label{appendix:rmse}
Observe that \begin{align}
&\Tr(\bM_0 {\bX}^{\mathrm{opt}})\!+\!\! \rho \sum_{j=1}^M \sigma_j^{-1}\!\left|\Tr\!\left(\bM_j({\bX}^{\mathrm{opt}}\!\!-\!\!\bv\bv^{*})\right)\right|\!\!-\!\!\rho f_{\mathrm{WLAV}}(\boldsymbol{\eta}) \notag  \\
&\stackrel{(a)}{\leq}\Tr(\bM_0{\bX}^{\mathrm{opt}})\!+\!\rho \sum_{j=1}^M \sigma_j^{-1}\left|\Tr\left(\bM_j({\bX}^{\mathrm{opt}}\!-\!\bv\bv^{*})\right)\!-\!\eta_j\right|  \notag \\
&\stackrel{(b)}{\leq} \Tr(\bM_0\bv\bv^{*}) + \rho f_{\mathrm{WLAV}}(\boldsymbol{\eta}), \label{bound}
\end{align}
where the relation (a) follows from a triangle inequality and the inequality (b) is obtained by evaluating the objective of \eqref{PSSE-SDPP2} at the feasible point $\bv\bv^{*}$.
Therefore, we have
\begin{align}
&\Tr\left(\bM_0({\bX}^{\mathrm{opt}}-\bv\bv^{*})\right) + \rho \sum_{j=1}^M \sigma_j^{-1}\left|\Tr\left(\bM_j({\bX}^{\mathrm{opt}}-\bv\bv^{*})\right)\right| \notag \\
&\leq\,  2\rho f_{\mathrm{WLAV}}(\boldsymbol{\eta})\, .
\label{bound2}
\end{align}
Recall that
\begin{align*}
\bM_{0} = \bH(\hat{\boldsymbol{\mu}}) - \sum_{j=1}^M\hat{\mu}_j\bM_{j}, \quad   \bH(\hat{\boldsymbol{\mu}})\bv = \mathbf{0}.
\end{align*}
Upon defining $\vartheta_j:=\Tr\left(\bM_j({\bX}^{\mathrm{opt}}-\bv\bv^{*})\right)$, one can write
\begin{align}
\!\!\!\!\sum_{j=1}^M (\rho \sigma_j^{-1}|\vartheta_j|-\hat{\mu}_j\vartheta_j)\!+\! \Tr(\bH(\hat{\boldsymbol{\mu}}){\bX}^{\mathrm{opt}})\!\leq\!
 2\rho f_{\mathrm{WLAV}}(\boldsymbol{\eta}).
\label{bound3}
\end{align}
Hence, it follows from \eqref{eq:rho} that
\begin{align}\label{boundX}
\Tr(\bH(\hat{\boldsymbol{\mu}}){\bX}^{\mathrm{opt}}) \leq 2\rho f_{\mathrm{WLAV}}(\boldsymbol{\eta}).
\end{align}
Now, consider the eigenvalue decomposition of $\bH(\hat{\boldsymbol{\mu}}) = \bU\bLambda\bU^{*}$,
where $\bLambda = \diag(\lambda_1,\ldots,\lambda_N)$ collects the eigenvalues of
$\bH(\hat{\boldsymbol{\mu}})$ that are sorted in descending order. The matrix $\bU$ is a unitary matrix whose columns are the corresponding eigenvectors.
Define
\begin{equation}
\check{\bX}:=
\begin{bmatrix}
\tbX     &\tbx  \\
\tbx^{*}   &\alpha
\end{bmatrix} = \bU^{*}{\bX}^{\mathrm{opt}}\bU,
\end{equation}
where $\tbX \in \mathbb{H}_{+}^{N-1}$ is the $(N\!-\!1)$-th order leading principal submatrix of $\check{\bX}$.
It can be concluded from \eqref{boundX} that
\begin{align*}
 2\rho f_{\mathrm{WLAV}}(\boldsymbol{\eta}) \geq
\Tr(\bH(\hat{\boldsymbol{\mu}}){\bX}^{\mathrm{opt}}) =  \Tr(\bLambda \check{\bX}) \geq \lambda_2\Tr(\tbX),
\end{align*}
where the last inequality follows from the fact that $\mathrm{rank}(\mathbf{H}(\hat{\boldsymbol{\mu}}))=N-1$.
Therefore, upper bounds for the trace and Frobenius norm of the matrix $\tbX$ can be obtained as
\begin{align*}
\|\tbX\|_F \leq \Tr(\tbX) \leq \frac{2\rho}{\lambda_2} f_{\mathrm{WLAV}}(\boldsymbol{\eta}).
\end{align*}
By defining $\tbv = {\bv}/{\|\bv\|_2}$, the matrix ${\bX}^{\mathrm{opt}}$ can be decomposed~as
\begin{align}\label{decompXhat}
{\bX}^{\mathrm{opt}} &=  \bU\check{\bX} \bU^{*} \notag =
\begin{bmatrix}
\tbU\;\;     \tbv
\end{bmatrix}
\begin{bmatrix}
\tbX     &\tbx  \\
\tbx^{*}   &\alpha
\end{bmatrix}
\begin{bmatrix}
\tbU^{*}  \\
\tbv^{*}
\end{bmatrix} \notag \\
&= \tbU\tbX\tbU^{*} + \tbv \tbx^{*}\tbU^{*} + \tbU\tbx\tbv^{*} + \alpha\tbv\tbv^{*}.
\end{align}
Since $\check{\bX}$ is positive semidefinite, the Schur complement dictates the relation
$\tbX - \alpha^{-1}\tbx\tbx^{*}\succeq \mathbf{0}$. Using the fact that
$\alpha = \Tr({\bX}^{\mathrm{opt}}) - \Tr(\tbX)$, one can write
\begin{align}\label{bound:xn}
\|\tbx\|_2^2 \leq \alpha\Tr(\tbX) = \Tr({\bX}^{\mathrm{opt}})\Tr(\tbX) - \Tr^2(\tbX).
\end{align}
Therefore,
\begin{subequations}
\begin{align}
\hspace{-0.3cm} \|{\bX}^{\mathrm{opt}} &- \alpha\tbv\tbv^{*}\|_F^2
= \|\tbU\tbX\tbU^{*}\|_F^2 + 2 \|\tbv \tbx^{*}\tbU^{*}\|_F^2 \label{err:FroDecom} \\
&= \|\tbX\|_F^2 + 2\|\tbx\|_2^2 \label{err:Unitary} \\
&\leq \|\tbX\|_F^2 - 2\Tr^2(\tbX)+ 2\Tr({\bX}^{\mathrm{opt}})\Tr(\tbX)  \label{err:trace}\\
&\leq 2\Tr({\bX}^{\mathrm{opt}})\Tr(\tbX) \label{err:trace2}\\
&\leq \frac{4\rho f_{\mathrm{WLAV}}(\boldsymbol{\eta}) }{\lambda_2}\Tr({\bX}^{\mathrm{opt}}),\label{err:end}
\end{align}
\end{subequations}
where \eqref{err:FroDecom} follows from the fact that $\tbU^{*}\tbv = \mathbf{0}$, \eqref{err:Unitary} is due to
$\tbU^{*}\tbU = \mathbf{I}_{N-1}$, and \eqref{err:trace2} is in light of $\|\tbX\|_F \leq \Tr(\tbX)$. The proof is completed by choosing $\beta$ as  $\alpha/\|\bv\|^2_2$.

\subsection{Proof of Corollary~\ref{coro:probbound}}\label{appendix:probbound}
Define $\tilde{\eta}_i:=\eta_i/\sigma_i$ for $i=1,\ldots,M$. Then, $\tilde{\boldsymbol{\eta}}$ is a standard Gaussian random vector and
\begin{align}	
f_{\mathrm{WLAV}}(\boldsymbol{\eta})=\|\tilde{\boldsymbol{\eta}}\|_1.
\end{align}
Applying the Chernoff's bound \cite{Boucheron} to the  Gaussian vector $\tbmeta$ yields that
\begin{align}
   \mathbb{P}(&\|\tbmeta\|_1>t) \leq \myexp^{-\psi t}\mathbb{E}\, \myexp^{\psi \|\tbmeta\|_1}\notag = \myexp^{-\psi t}(\mathbb{E}\, \myexp^{\psi |\tilde{\eta}_1|})^M \notag \\
   &= \myexp^{-\psi t}\left(\sqrt{\frac{2}{\pi}}\int_0^{\infty}\myexp^{\psi x}\myexp^{-x^2/2}\diff x\right)^M \notag \\
   &= \myexp^{-\psi t}\left(\myexp^{\frac{\psi^2}{2}}\mathrm{erfc}\left(\frac{-\psi}{\sqrt{2}}\right)\right)^M
   \leq 2^M \myexp^{(M\psi^2-2\psi t)/2}\; , \label{upboundL1}
\end{align}
which holds for every $\psi>0$. Note that the complementary error function $\mathrm{erfc}(a) := \frac{2}{\sqrt{\pi}}\int_a^{\infty}\myexp^{-x^2}\diff x \leq 2$ holds for all $a\in \mathbb{R}$.
The minimization of the upper bound \eqref{upboundL1} with respect to $\psi$  gives  the equation $\psi^{\mathrm{opt}} =\frac{t}{M}$. Now, it follows from Theorem~\ref{thm:rmse}  that
\begin{align} \label{probbound}
\mathbb{P}(\zeta>t) &\leq \mathbb{P}\left(2\sqrt{\frac{\rho \|\tbmeta\|_{1} }{N\lambda}}>t\right) = \mathbb{P}\left(\|\tbmeta\|_{1}> \frac{t^2 N\lambda}{4\rho}\right) \notag \\
&\leq \exp\left(M\ln2-\frac{t^4 N^2\lambda^2}{32M\rho^2} \right).
\end{align}
The proof is completed by substituting $N = \frac{M}{\kappa}$ into \eqref{probbound}.

\bibliographystyle{IEEEtran}
\bibliography{SE_Line_bib}

\end{document}